\documentclass[10pt,a4paper]{article}
\usepackage[top=1in, bottom=1in, left=1in, right=1in]{geometry}
\usepackage[utf8]{inputenc}
\usepackage{amsmath}
\usepackage{amsthm}
\usepackage{amsfonts}
\usepackage{amssymb}
\usepackage{graphicx}
\usepackage[all]{xy}
\usepackage{tipa}
\theoremstyle{plain}
\newtheorem{thm}{Theorem}[section]
\newtheorem{cor}[thm]{Corollary}
\newtheorem{lem}[thm]{Lemma}
\newtheorem{prop}[thm]{Proposition}
\theoremstyle{definition}
\newtheorem{defn}[thm]{Definition}
\theoremstyle{remark}
\newtheorem{rem}[thm]{Remark}
\newtheorem{ex}[thm]{Example}
\newcommand{\oDnt}{\mathcal{D}_n ^\theta}
\newcommand{\Dnt}{D_n^\theta}
\newcommand{\oDMt}{\mathcal{D}^{\theta}(M)}
\newcommand{\oDntL}{{\mathcal{D}_n ^\theta \times \mathcal{L}}}
\newcommand{\DntL}{{D_n ^\theta \times L}}
\newcommand{\DMtL}{{D^\theta (M) \times L}}

\newcommand{\Map}{\operatorname{Map}}
\author{Inbar Klang}
\title{The factorization theory of Thom spectra and twisted non-abelian Poincar\'e duality}
\begin{document}
\maketitle 
\let\thefootnote\relax\footnote{Inbar Klang is supported by a Gabilan Stanford Graduate Fellowship, Department of Mathematics, Stanford University.}

\begin{abstract}
We give a description of the factorization homology and $E_n$ topological Hochschild cohomology of Thom spectra arising from $n$-fold loop maps $f: A \to BO$, where $A = \Omega^n X$ is an $n$-fold loop space. We describe the factorization homology $\int_M Th(f)$ as the Thom spectrum associated to a certain map $\int _M A \to BO$, where $\int _M A$ is the factorization homology of $M$ with coefficients in $A$. When $M$ is framed and $X$ is $(n-1)$-connected, this spectrum is equivalent to a Thom spectrum of a virtual bundle over the mapping space $\Map_c(M,X)$; in general, this is a Thom spectrum of a virtual bundle over a certain section space. This can be viewed as a twisted form of the non-abelian Poincar\'e duality theorem of Segal, Salvatore, and Lurie, which occurs when $f: A \to BO$ is nullhomotopic. This result also generalizes the results of Blumberg-Cohen-Schlichtkrull on the topological Hochschild homology of Thom spectra, and of Schlichtkrull on higher topological Hochschild homology of Thom spectra. We use this description of the factorization homology of Thom spectra to calculate the factorization homology of the classical cobordism spectra, spectra arising from systems of groups, and the Eilenberg-MacLane spectra $H \mathbb{Z}/p$, $H\mathbb{Z}_{(p)}$, and $H\mathbb{Z}$. We build upon the description of the factorization homology of Thom spectra to study the ($n=1$ and higher) topological Hochschild cohomology of Thom spectra, which enables calculations and a description in terms of sections of a parametrized spectrum. If $X$ is a closed manifold, Atiyah duality for parametrized spectra allows us to deduce a duality between $E_n$ topological Hochschild homology and $E_n$ topological Hochschild cohomology, recovering string topology operations when $f$ is nullhomotopic. In conjunction with the higher Deligne conjecture, this gives $E_{n+1}$ structures on a certain family of Thom spectra, which were not previously known to be ring spectra.
\end{abstract}

\section{Introduction}
In this paper, we study the factorization homology and $E_n$ topological Hochschild cohomology of Thom spectra. Factorization homology has received a considerable amount of interest recently, in large part due to its connection to topological field theories and to configuration space models for mapping spaces. Fixing an $E_n$-algebra $A$, factorization homology $\int_{-}A$ with coefficients in $A$ satisfies the generalized Eilenberg-Steenrod axioms, described by Ayala-Francis in \cite{AF}, of a homology theory for $n$-manifolds. In fact, Ayala-Francis show that all such homology theories arise as $\int_{-} A$ for some $E_n$-algebra. Factorization homology then forms an important class of topological field theories: the ones in which the global observables are determined by the local observables. This is because the Ayala-Francis axioms for factorization homology imply that $\int_M A$, the value of the field theory on the manifold $M$, is determined by $\int_{\mathbb{R}^n} A$ and patching data for $M$.

Another way in which factorization homology generalizes ordinary homology is that it can be modeled using labeled configuration spaces; in fact, it originates from configuration space models for mapping spaces, as in \cite{McD}. If $A$ is a discrete abelian group, $\int_M A$ is the labeled configuration space $A[M]$ from the Dold-Thom theorem, whose homotopy groups are $H_*(M;A)$. For $A$ a more general $E_n$-algebra in topological spaces, Segal in \cite{Seg10} and Salvatore in \cite{Sal} considered configuration spaces with amalgamation, or configuration spaces with summable labels. These are configurations of points in $M^n$ labeled by elements of $A$, with labels combining when points ``collide". When $M$ is framed and $\pi_0(A)$ is a group, this amalgamated configuration space is equivalent to the space of compactly supported maps $\Map_c(M,B^n A)$. This equivalence, attributed to Segal, Salvatore, and Lurie, is called non-abelian Poincar\'e duality, as it reduces to ordinary Poincar\'e duality when $A$ is abelian. The connection between factorization homology and configuration spaces has recently proven to be very fruitful in Knudsen's work \cite{Knu} on rational homology of unordered configuration spaces.

\paragraph{Previous work.} Factorization homology can be difficult to compute, particularly when the $E_n$-algebras are valued in spectra or chain complexes. Descriptions of the factorization homology of free $E_n$-algebras and of $E_n$-enveloping algebras of Lie algebras are known; see Section 5 of \cite{AF} and Section 3 of \cite{Knu}. Suspension spectra provide another class of algebras for which factorization homology is known. This follows from non-abelian Poincar\'e duality, which provides a description of the factorization homology of $n$-fold loop spaces in terms of mapping spaces or section spaces. Factorization homology commutes with the suspension spectrum functor; that is, if $A$ is an $E_n$-space, $\int_M \Sigma^\infty _+ A \simeq \Sigma^\infty _+ \int_M A$. This gives a description of the factorization homology of suspension spectra: $\int_M \Sigma^\infty _+ \Omega^n X \simeq \Sigma^\infty _+ \Map_c(M,X)$ if $M$ is framed. For $M=S^1$, factorization homology specializes to topological Hochschild homology. Thus this description of factorization homology of suspension spectra recovers B\"okstedt and Waldhausen's result relating $THH$ to the free loop space. When $X$ is a closed manifold, Atiyah duality for parametrized spectra can then be used to describe topological Hochschild cohomology of $\Sigma^\infty _+ \Omega X$ (see, e.g., \cite{Mal}): $THC(\Sigma^\infty _+ \Omega X) \simeq LX^{-TX}$.

The goal of this paper is to describe and compute the factorization homology and the $E_n$ Hochschild cohomology of Thom spectra. The Thom spectrum of an $n$-fold loop map to $BO$ or $BGL_1(R)$, which we denote $Th(\Omega^n f)$ or $\Omega^n X^{\Omega^n f}$, is an $E_n$-ring spectrum by a theorem of Lewis (Theorem 9.7.1 of \cite{LMS}). In \cite{BCS}, Blumberg, Cohen, and Schlichtkrull study topological Hochschild homology of Thom spectra, expressing $THH(\Omega X ^{\Omega f})$ as a Thom spectrum of a virtual bundle over $LX$. Factorization homology of $E_{\infty}$-ring spectra agrees with higher topological Hochschild homology (Theorem 5 of \cite{GTZ14} or Proposition 5.1 of \cite{AF}), and in \cite{Sch}, Schlichtkrull describes the higher topological Hochschild homology of Thom spectra of infinite loop maps. In one sense, this is more general than factorization homology, as higher topological Hochschild homology is defined over any CW complex rather than just manifolds, and Schlichtkrull's result therefore applies to any CW complex as well. This, however, does not address Thom spectra of $n$-fold loop maps for $n < \infty$, and there is little known about the topological Hochschild cohomology (higher or otherwise) of Thom spectra.

Factorization homology is a higher-dimensional generalization of Hochschild homology, and the corresponding generalization of Hochschild cohomology is higher Hochschild cohomology. Higher (topological) Hochschild cohomology, $(T)HC _{E_n}(A)$, is an invariant of $E_n$-algebras which naturally extends Hochschild cohomology of algebras. In analogy with Hochschild cohomology, it is important for studying deformations of $E_n$-algebras, and for understanding $E_{n+1}$-structures, for example, those appearing in string topology. Hochschild cohomology of $A$ is a derived mapping object $Rhom_{A\text{-}bimod}(A,A)$ of $A$-bimodules, and higher Hochschild cohomology of an $E_n$-algebra $A$ is similarly $Rhom_{E_n\text{-}A}(A,A)$, derived maps of $E_n$-$A$-modules, see, e.g., Section 3 of \cite{Fra}, Section 2 of \cite{HKV}, or Section 3 of \cite{Hor14} for a definition. For $n=1$, the category of $E_n$-$A$-modules is equivalent to the category of $A$-bimodules. There is an alternate useful description of higher Hochschild homology in terms of maps of left $\int_{S^{n-1} \times \mathbb{R}} A$-modules, $Rhom_{\int_{S^{n-1} \times \mathbb{R}} A}(A,A)$ (see, e.g., Proposition 3.16 of \cite{Fra} or Proposition 3.19 of \cite{Hor16}). This connection with factorization homology was used by Francis in \cite{Fra} to study the tangent complex of $E_n$-algebras, by Ginot, Tradler, and Zeinalian in \cite{GTZ12} to study higher string topology, and by Horel in \cite{Hor16} to prove an \'etale base change theorem for higher topological Hochschild cohomology.

\paragraph{Summary of results.} Our main result about factorization homology of Thom spectra expresses $\int_{M^n} \Omega^n X^{\Omega^n f}$ as a Thom spectrum of a virtual bundle over a mapping or section space, and for framed $M$, we give an explicit map $\Map_c(M, X) \to BO$  whose Thom spectrum is $\int_M \Omega^n X^{\Omega^n f}$, see Theorem \ref{thm-main} (or Theorem \ref{thm-fact-hom-gen}, for generalized Thom spectra). This is a direct generalization of the description in \cite{BCS}. This result can be seen as a twisting of non-abelian Poincar\'e duality: viewing a Thom spectrum as a twisted suspension spectrum, it shows that the factorization homology of a twisted suspension spectrum of $\Omega^n X$ is a twisted suspension spectrum of $\Map_c(M,X)$, or of its section space counterpart. In this paper, we derive two main uses from this result. In Sections 3 and 4, we calculate the factorization homology of cobordism spectra (Corollary \ref{cor-lie-gps} and \ref{cor-sys-gps}), recovering Schlichtkrull's calculations in \cite{Sch}, and of the Eilenberg-MacLane spectra $H\mathbb{Z}/p$, $H\mathbb{Z}_{(p)}$, and $H\mathbb{Z}$ over oriented surfaces (Propositions \ref{prop-HZ/2}, \ref{prop-HZ/p}, \ref{prop-HZ(p)}, and Corollary \ref{cor-HZ}), recovering and generalizing calculations of higher topological Hochschild homology from \cite{Vee}, \cite{BLKRZ}, and \cite{DLR}, although we do not determine the multiplicative structure. In Section 5, we turn to cohomology, building upon our description of the factorization homology of Thom spectra to develop a description of the higher topological Hochschild cohomology of Thom spectra. This requires additional techniques, the main ingredient of which is an action of $\Sigma^\infty _+ \Omega X$ on $\Omega^n X^{\Omega^n f}$.

The loop space $\Omega X$ acts on itself by conjugation, and there are equivalences

$$THC(\Sigma^\infty _+ \Omega X) \simeq Rhom_{\Sigma^\infty _+ \Omega X}(S, (\Sigma^\infty _+ \Omega X)^{ad}), \thinspace THH(\Sigma^\infty _+ \Omega X) \simeq S \wedge^L _{\Sigma^\infty _+ \Omega X} (\Sigma^\infty _+ \Omega X)^{ad}$$

\noindent (see, e.g., Section 4 of \cite{Mal}). $(-)^{ad}$ denotes the conjugation action. Our main results about higher Hochschild cohomology of Thom spectra (Lemma \ref{lem-ring-map} and Theorem \ref{thm-conj-rhom}) generalize this to a conjugation action of $\Omega X$ on $\Omega^n X^{\Omega^n f}$, and give an analogous description of the higher Hochschild homology and cohomology in terms of this action.

\paragraph{Theorem.} $$THC _{E_n}(\Omega^n X^{\Omega^n f}) \simeq Rhom_{\Sigma^\infty _+ \Omega X}(S,\Omega^n X^{\Omega^n f})$$

and

$$THH^{E_n}(\Omega^n X^{\Omega^n f}) \simeq S \wedge^L _{\Sigma^\infty _+ \Omega X} \Omega^n X^{\Omega^n f}$$

The action of $\Sigma^\infty _+ \Omega X$ on $\Omega^n X^{\Omega^n f}
$ is fairly tractable -- in particular, it is homotopically trivial if $f$ is an $E_1$-map -- so this description allows us to compute $E_n$ topological Hochschild cohomology of cobordism spectra and of certain Eilenberg-MacLane spectra, and topological Hochschild cohomology of the Ravenel spectra $X(n)$, see Section 5.

Importantly, this description also implies that $THC_{E_n} (\Omega^n X^{\Omega^n f})$ is the cohomology (that is, section spectrum) of a parametrized spectrum over $X$, with fiber spectrum $\Omega^n X^{\Omega^n f}$. The homology of this parametrized spectrum is the higher topological Hochschild homology, $THH^{E_n}(\Omega^n X^{\Omega^n f})$, which agrees with $\int_{S^n \times \mathbb{R}} \Omega^n X^{\Omega^n f}$ if $\Omega^n X^{\Omega^n f}$ is $E_{n+1}$, see Section 3.2 of \cite{Fra}. If $X$ happens to be a closed manifold, as is the case in (higher) string topology, Atiyah duality for parametrized spectra gives:

\paragraph{Corollary.} If $X$  is a closed manifold and $THH^{E_n}(\Omega^n X ^{\Omega^n f}) \simeq \Map(S^n,X)^{l^n(f)}$, then
$$THC _{E_n}(\Omega^n X ^{\Omega^n f}) \simeq \Map(S^n, X)^{l^n(f)-TX}$$

\paragraph{} By our description of factorization homology of Thom spectra, there is such a map $l^n(f)$ if $THH^{E_n}$ and $\int_{S^n}$ agree. (This is not always the case; see Section 5 for details.) For $n=1$, $l(f)$ can be taken to be the map of Blumberg-Cohen-Schlichtkrull \cite{BCS}, which gives $THH(\Omega X^{\Omega f})$.

If $A$ is an $E_n$-ring spectrum, then by the higher Deligne conjecture, $THC_{E_n}(A)$ is an $E_{n+1}$-ring spectrum. Thus (Corollary \ref{cor-En+1}):

\paragraph{Corollary.} If $X$ is a closed manifold and $THH^{E_n}(\Omega^n X ^{\Omega^n f}) \simeq \Map(S^n,X)^{l^n(f)}$, $\Map(S^n, X)^{l^n(f)-TX}$ is an $E_{n+1}$-ring spectrum.

\paragraph{} If $f$ is nullhomotopic, this recovers the fact that $\Map(S^n, X)^{-TX}$ is an $E_{n+1}$-ring spectrum, which gives string topology operations on the homology of the free loop space when $n=1$ (see \cite{CJ02}), and higher string topology operations on $H_*(\Map(S^n,X))$ for higher $n$ (see \cite{Hu}, \cite{GS}). In this generality, this result is new. The case $n=1$ answers a question of T. Kragh, and we thank him for his interest in this project.

\paragraph{Methods.} As in \cite{BCS} and \cite{Sch}, our description of the factorization homology of Thom spectra ultimately follows from multiplicative properties of the Thom spectrum functor. Whereas the description directly generalizes to factorization homology, the methods are somewhat different, as the cyclic bar construction and Loday functor used in \cite{BCS} and \cite{Sch} respectively are not available for factorization homology in general. The development of factorization homology, both as a homology theory for manifolds in the Ayala-Francis axiomatic framework and as a monadic two-sided bar construction as in \cite{KM} and \cite{Mil}, and the theory of generalized Thom spectra developed in \cite{ABGHR} and \cite{ABG}, allow us to generalize the description in \cite{BCS} to factorization homology, while doing away with most of the technical difficulty. Although it is not a deep theorem, this description of the factorization homology of Thom spectra allows for calculations and for insight into the higher topological Hochschild cohomology of Thom spectra.

Our description of $\int_M \Omega^n X^{\Omega^n f}$ as a Thom spectrum can be obtained either from the point-set machinery, using the Lewis-May Thom spectrum (see Chapter 9 of \cite{LMS}) and the two-sided bar construction model of factorization homology, or from the $\infty$-categorical machinery, using the theory of Thom spectra from \cite{ABGHR} and the axiomatic framework for factorization homology. For conceptual reasons, much of the work is done using the point-set machinery. The monadic two-sided bar construction $B(D(M),D_n,A)$, which comes with an explicit scanning map to $\Map_c(M,B^n A)$ when $A$ is an $E_n$-space and $M$ is framed, makes clearest the relation to non-abelian Poincar\'e duality, and is conducive to describing the maps $\Map_c(M,B^n A) \to BO$ explicitly. In addressing generalized Thom spectra, the $\infty$-categorical machinery is most efficient.

In order to obtain our description of higher topological Hochschild cohomology of Thom spectra, we construct a ring map $\Sigma^\infty _+ \Omega X \to \int_{S^{n-1} \times \mathbb{R}} \Omega^n X^{\Omega^n f}$ and study the resulting action of $\Sigma^\infty _+ \Omega X$ on $\Omega^n X^{\Omega^n f}$. This allows for computation and for an interpretation via parametrized spectra and Atiyah duality.

\paragraph{Structure of the paper.} In Section 2, we introduce necessary preliminaries on Lewis-May Thom spectra and on factorization homology. This includes brief expositions on both the two-sided bar construction model and the axiomatic approach. In Section 3, we use operadic properties of the Lewis-May Thom spectrum functor, which ensure that it behaves well with respect to the two-sided bar construction, to describe $\int_M \Omega^n X^{\Omega^n f}$ as a Thom spectrum (Theorem \ref{thm-main}). This includes, for framed $M$, an explicit map $\Map_c(M, X) \to BO$ whose Thom spectrum is $\int_M \Omega^n X^{\Omega^n f}$. We then deduce calculations of factorization homology of several important $E_{\infty}$-ring spectra over stably framed manifolds. We also use the Thom isomorphism theorem to give an explicit calculation of $\int_M H\mathbb{Z}/2$ for $M$ an orientable surface. In Section 4, we use the $\infty$-categorical approach to generalized Thom spectra from \cite{ABGHR}, along with the axiomatic description of factorization homology of \cite{AF}, to expand the results of Section 3 to generalized Thom spectra (Theorem \ref{thm-fact-hom-gen}). We then use a Thom isomorphism argument to calculate $\int_M H\mathbb{Z}/p$, $\int_M H\mathbb{Z}_{(p)}$, and $\int_M H\mathbb{Z}$, for $M$ an orientable surface. In Section 5, we describe the higher topological Hochschild cohomology of Thom spectra as a derived mapping spectrum of $\Sigma^\infty _+ \Omega X$-modules (Theorem \ref{thm-conj-rhom}). Via parametrized spectra, we relate higher topological Hochschild homology and cohomology of Thom spectra to Atiyah duality and string topology, which results in new $E_{n+1}$-ring spectra. Theorem \ref{thm-conj-rhom} also gives calculations of (higher) topological Hochschild cohomology of Thom spectra.

\paragraph{Acknowledgments.} This paper represents a part of my Stanford University Ph.D. thesis, and I am grateful to my advisor, Ralph Cohen, for his guidance and support throughout this project. I would like to thank Andrew Blumberg for helpful comments on an earlier draft. I would also like to thank Alexander Kupers, Oleg Lazarev, and Jeremy Miller for helpful conversations. I thank the anonymous referee for their careful reading and constructive input.

\section{Preliminaries}
\subsection{The Lewis-May Thom spectrum functor}
In Section 3, we will use the Lewis-May Thom spectrum functor, which has nice operadic properties that ensure it ``commutes" with factorization homology. We briefly describe this functor; for more details, see Chapter 9 of \cite{LMS}.

Let $U$ be a real inner product space of countably infinite dimension. For $V$ a finite dimensional subspace of $U$, denote by $O(V)$ its group of orthogonal transformations, with classifying space $BO(V)$. Denote
$$BO = colim (BO(V))$$
\noindent where the colimit runs over finite-dimensional linear subspaces of $U$.

Let $X$ be a compactly generated, weak Hausdorff space. To a map $f:X \to BO$, the Lewis-May Thom spectrum functor associates a spectrum indexed on finite-dimensional subspaces of $U$, that is, a set of spaces $E(V)$ with structure maps 
$$\Sigma^W E(V) \to E(W \oplus V)$$
\noindent satisfying certain properties. Denote by $S^V$ the 1-point compactification of $V$.

For a map $f: X \to BO$, denote $X(V) = f^{-1}(BO(V))$. Assemble a prespectrum $T(f)$ from the Thom spaces of the maps $X(V) \to BO(V)$: let $\xi(V)$ be the spherical bundle over $X(V)$ obtained by pulling back the canonical $S^V$-bundle over $BO(V)$ along $f: X(V) \to BO(V)$. Denote by $T(f)(V)$ the Thom space of this spherical bundle; that is, $T(f)(V)$ is obtained from the total space of $\xi(V)$ by collapsing the section at $\infty$ to a point. The structure maps of $T(f)$ are induced by the pullback diagrams
$$\xymatrix{
\Sigma^W_{X(V)} \xi(V) \ar[r] \ar[d] & \xi(W \oplus V) \ar[d] \\
X(V) \ar[r] & X(W \oplus V)
}$$

\noindent Here $\Sigma^W_{X(V)}$ denotes fiberwise suspension.

The Thom spectrum $Th(f)$ is defined to be the spectrification of the Lewis-May prespectrum $T(f)$.

\begin{rem}\label{rem-LMS-properties}
We will often rely on the fact that a weak equivalence over $BO$ (or $BF$, if the maps to $BF$ are ``good"), induces an equivalence of Thom spectra. This also implies that Thom spectra of homotopic maps are equivalent.
\end{rem}

\subsection{Factorization homology}
This subsection describes the two-sided bar construction model of factorization homology, as well as the Ayala-Francis axiomatic characterization of factorization homology as a homology theory for manifolds, and discusses non-abelian Poincar\'e duality, which relates factorization homology in the category of topological spaces to mapping spaces and section spaces.

\subsubsection{Operads and monads}
We now briefly recall some basic properties of operads and monads, and give important examples. We refer the reader to Sections 1 and 2 of \cite{May72} for more details.

Let $(\mathfrak{C}, \otimes, I)$ be a cocomplete symmetric monoidal category, with product $\otimes$ and unit element $I$. Let $\mathfrak{C}^{\Sigma}$ denote the category of symmetric sequences in $\mathfrak{C}$, that is, sequences $\mathcal{C}=(\mathcal{C}(n))$ where each $\mathcal{C}(n)$ is equipped with an action of the symmetric group $\Sigma_n$. Define a composition monoidal product on $\mathfrak{C}^{\Sigma}$ by
$$(\mathcal{C} \circ \mathcal{D})(n) = \coprod_{k, n_1 + ... +n_k =n} \mathcal{C}(k)\otimes_{\Sigma_k} (\mathcal{D}(n_1) \otimes ... \otimes \mathcal{D}(n_k))\otimes_{\Sigma_{n_1} \times ... \times \Sigma_{n_k}} (\amalg_{n!} I)$$

\begin{defn}\label{def-operad}
An operad in $\mathfrak{C}$ is a monoid in the category $\mathfrak{C}^{\Sigma}$ of symmetric sequences with the composition product defined above.
\end{defn}

Roughly speaking, an operad $\mathcal{O}$ consists of a symmetric sequence $(\mathcal{O}(n))$ with maps
$$\mathcal{O}(k)\otimes (\mathcal{O}(n_1) \otimes ... \otimes \mathcal{O}(n_k)) \to \mathcal{O}(n_1 + ... + n_k)$$
\noindent satisfying equivariance and associativity conditions, with a unit map $I \to \mathcal{O}(1)$.

\begin{ex}\label{ex-operads}
We will mainly be interested in the following operads in the category of topological spaces (with cartesian product, and unit the one-point space $*$):
\begin{itemize}
\item Let $U$ be an inner product space of countably infinite dimension. Define the linear isometries operad $\mathcal{L}$ by $\mathcal{L}(n) = Isom(U^n,U)$, the space of linear isometric embeddings $U^n \to U$, with composition product given by multicomposition of linear embeddings and unit $id: U \to U$. The linear isometries operad $\mathcal{L}$ is an $E_{\infty}$ operad, that is, all of its spaces are weakly equivalent to a point.

\item Little discs operads. Let $\theta: B \to BO(n)$ be a fibration; a $\theta$-framing of a manifold $M^n$ is a lift of the classifying map of its tangent bundle over $\theta$. Examples of this are orientation, Spin structure, or tangential framing (a trivialization of the tangent bundle). Note that $\mathbb{R}^n$ has a $\theta$-framing for all nonempty $B$; for each $B$, we fix such a $\theta$-framing (for example, coming from the standard tangential framing of $\mathbb{R}^n$). We will consider the $\theta$-framed little discs operad, $\oDnt$, whose $k^{th}$ space $\oDnt (k)= Emb^{\theta}(\coprod_k \mathbb{R}^n, \mathbb{R}^n)$ is roughly the space of embeddings $\coprod_k \mathbb{R}^n \hookrightarrow \mathbb{R}^n$ preserving $\theta$-framing. For a precise definition, see Section 2 of \cite{KM}. Composition is by multicomposition of embeddings, and the unit is the identity map of $\mathbb{R}^n$. For $B= EO(n)$ contractible, a $\theta$-framing is a tangential framing, and we will denote the corresponding little discs operad by $\mathcal{D}_n$. If $\theta: BG \to BO(n)$ is induced by a continuous group homomorphism $G \to O(n)$, we will sometimes denote $\oDnt$ by $\mathcal{D}_n^G$. The operad $\mathcal{D}_n^G$ is also the semidirect product of $\mathcal{D}_n$ with the group $G$; for a definition of this semidirect product, see Section 2 of \cite{SW}.
\end{itemize}

\end{ex}

\begin{defn}\label{def-alg-operad}
Let $\mathcal{O}$ be an operad in $\mathfrak{C}$. An algebra over $\mathcal{O}$ is the arity 0 component $A$ of a left module over $\mathcal{O}$, of the form $(A, \emptyset, \emptyset,...)$, in the category of symmetric sequences. Roughly, this structure endows $A$ with a unit map $\nu: I \to A$ and operations
$$\zeta_n: \mathcal{O}(n) \otimes A^{\otimes n} \to A$$
\noindent satisfying unitality, associativity and equivariance conditions.
\end{defn}

An algebra $A$ over a little discs operad or an $E_{\infty}$ operad in topological spaces is called grouplike if $\pi_0(A)$ is a group (under the multiplication induced by $\mathcal{O}(2) \times A^2 \to A$). In particular, a connected algebra is grouplike.

\begin{ex}\label{ex-algebras}
\begin{itemize}
\item The space $BO$ is an algebra over $\mathcal{L}$, as are classifying spaces of other stabilized Lie groups. The operad $\mathcal{L}$ is an $E_{\infty}$ operad, so a grouplike algebra over the linear isometries operad $\mathcal{L}$ is an infinite loop space by Theorem 14.4 of \cite{May72}.

\item For any pointed space $X$, $\Omega^n X$ is a $\mathcal{D}_n$-algebra by plugging the compactly supported maps $\mathbb{R}^n \to X$ into the embeddings $\coprod_k \mathbb{R}^n \hookrightarrow \mathbb{R}^n$ (sending any point in $\mathbb{R}^n$ outside the image of these embeddings to the basepoint). Conversely, May's recognition principle (\cite{May72}, Theorem 13.1) states that any grouplike $\mathcal{D}_n$-algebra is weakly equivalent as a $\mathcal{D}_n$-algebra to an $n$-fold loop space.

\item For any pointed $G$-space $X$ and continuous homomorphism $G  \to O(n)$, $\Omega^n X$ is a $\mathcal{D}_n^G$-algebra, as $\mathcal{D}_n^G$ is the semidirect product $\mathcal{D}_n \rtimes G$. Embeddings of discs act as above, and $G$ acts on $X$, and on the loop coordinate by rotation and reflection of loops (via $G \to O(n)$). Conversely, the equivariant recognition principle of Salvatore-Wahl states that any grouplike $\mathcal{D}_n^G$-algebra is weakly equivalent as a $\mathcal{D}_n^G$-algebra to an $n$-fold loop space on a pointed $G$-space, see Theorem 3.1 of \cite{SW}.
\end{itemize}
\end{ex}

\bigskip

An operad in topological spaces defines a monad on topological spaces and on spectra.

\paragraph{Monads from operads in topological spaces.} Given an operad $\mathcal{P}$ in topological spaces, one has an associated monad $P$ on topological spaces given by
$$PX = \coprod_n \mathcal{P}(n) \times_{\Sigma_n} X^n$$

An operad $\mathcal{P}$ with a map to the linear isometries operad $\mathcal{L}$ also defines a monad $P$ on Lewis-May spectra using the twisted half-smash product, see, e.g., Chapter 7 of \cite{LMS}. Briefly, each linear isometry $f: U^n \to U$ gives a way to internalize the external smash product from spectra indexed on $U^n$ to spectra indexed on $U$; as a result, one gets a twisted half-smash product $\mathcal{L}(n) \ltimes (X_1 \wedge ... \wedge X_n)$. If $A$ is a space equipped with a map to $\mathcal{L}(n)$, the twisted half-smash product $A \ltimes (X_1 \wedge ... \wedge X_n)$ is defined, and thus if $\mathcal{P}$ is a topological operad with a map to $\mathcal{L}$, we can use the twisted half-smash products $\mathcal{P}(n) \ltimes X^{\wedge n}$ to form $PX = \bigvee_n \mathcal{P}(n) \ltimes_{\Sigma_n} X^{\wedge n}$. As $\mathcal{P}$ is an operad, this is a monad. We say that a spectrum $E$ is an $\mathcal{P}$-algebra if it is a $P$-algebra.

\begin{defn}\label{def-right-modules}

\begin{itemize}
\item A right module over an operad $\mathcal{O}$ is a symmetric sequence $\mathcal{R}$ with a map $\mathcal{R} \circ \mathcal{O} \to \mathcal{R}$ satisfying associativity. That is, a collection of maps 
$$\mathcal{R}(k)\otimes (\mathcal{O}(n_1) \otimes ... \otimes \mathcal{O}(n_k)) \to \mathcal{R}(n_1 + ... + n_k)$$
\noindent satisfying equivariance and associativity properties.

\item A right functor over a monad $O$ is a functor $R: \mathfrak{C} \to \mathfrak{C}$ with a natural transformation $RO \to R$, satisfying associativity.
\end{itemize}
\end{defn}

As before, in both spaces and spectra, a right module over an operad defines a right functor over a monad. For details, see, e.g., Section 2 of \cite{Mil}.

\begin{ex}\label{ex-right-modules}
\begin{itemize}
\item Any operad is a right module over itself, with action map given by the operad structure map $\mathcal{O} \circ \mathcal{O} \to \mathcal{O}$.

\item Let $M^n$ be a $\theta$-framed manifold. Denote by $\oDMt$ the right module over $\oDnt$ whose $k^{th}$ space $\oDMt(k)= Emb^{\theta}(\coprod_k \mathbb{R}^n, M)$ is roughly the space of embeddings $\coprod_k \mathbb{R}^n \hookrightarrow M$ preserving the $\theta$-framing. The module structure is given by multicomposition of embeddings. For a precise definition, see Section 2 of \cite{KM}. This right module gives a right functor $D^{\theta}(M)$ over the monad $\Dnt$.

\item On a pointed space $X$, the reduced monad $\mathbb{D}_n X$ is obtained from $D_n X$ by a basepoint relation, see Notations 2.3 and Construction 2.4 of \cite{May72}. For a $D_n$-algebra $A$, we take its basepoint to be the unit. $D_n$ and $\mathbb{D}_n$ are related by $D_n X = \mathbb{D}_n(X_+)$. Reduced versions of $D_n^\theta$ and the right functor $D^\theta(M)$ can also be defined using a basepoint relation.

\item The $n$-fold suspension $\Sigma^n$ is a right functor over $\mathbb{D}_n$ by ``scanning"; for a pointed space $X$, the map $\Sigma^n \mathbb{D}_n X \to \Sigma^n X$ is defined as follows:

For $P= (t,f_1,...f_n,x_1,...x_n)$, if $t$ is not in the image of any of the embeddings $f_i$, take $P$ to the basepoint. Else, if $t \in im(f_i)$, take $P$ to $(f_i^{-1}(t),x_i)$.

Similarly, $\Sigma^n_+$ is a right functor over $D_n$.
\end{itemize}

\end{ex}

\subsubsection{The two-sided monadic bar construction}
Let $O$ be a monad coming from a well-behaved operad in topological spaces (that is, satisfying the Cofibration Hypothesis of \cite{EKMM} VII.4), $R$ a right module over $O$, and $A$ an $O$-algebra. If $\mathcal{O}(0) = *$, we require that the unit of $A$ is a nondegenerate basepoint. If $A$ is a space, we require that $A$ is homotopy equivalent to a cofibrant space (e.g., a cell complex), and if $A$ is a spectrum, we require that it is homotopy equivalent to a cofibrant spectrum (e.g., a cellular spectrum). Then, by Proposition 5.8 of \cite{ABGHR2}, $B(O,O,A)$ is homotopy equivalent to a cofibrant $O$-algebra and thus a derived, or homotopy-invariant, version of the tensor product
$$\xymatrix{
R\otimes_O A =coeq(ROA \ar@<-0.5ex>[r] \ar@<0.5ex>[r] & RA)
}$$
\noindent is given by the two-sided bar construction, $B(R,O,A)$, which is the geometric realization of the simplicial object $B_i(R,O,A)= R O^i A$. $O^iA$ denotes the $i$-fold iteration of the monad $O$ applied to $A$, and $R O^i A$ is obtained by applying $R$ to $O^i A$. The face maps are given by compositions ($RO \to R$, $O^2 \to O$, and $OA \to A$), and the degeneracies by units $Id \to O$.

\begin{ex}\label{ex-delooping}
If $A$ is a grouplike $\mathcal{D}_n$-algebra, we define its $n$-fold delooping $B^n A$, by $B(\Sigma^n_+, D_n, A)$. This is equivalent to the May model $B(\Sigma^n, \mathbb{D}_n, A)$; see Section 5 of \cite{KM} for more details about this model.
\end{ex}

\begin{defn}\label{def-DML} Let $\DMtL$ denote the functor on spaces or spectra associated to the right $\oDntL$-module $\mathcal{D}^\theta (M) \times \mathcal{L}$. For pointed spaces, we denote its reduced version by $\mathbb{D}^\theta(M) \times L$.
\end{defn}

\begin{defn}\label{def-factorization-homology}
Let $A$ be a $\oDntL$-algebra (in spaces or spectra) whose underlying space or spectrum is cofibrant, and let $M^n$ be a $\theta$-framed $n$-manifold. Define the factorization homology of $A$ over $M$ by
$$\int_M A = B(\DMtL, \DntL, A)$$

\noindent If $A$ is not underlying cofibrant, one can first cofibrantly replace $A$.
\end{defn}

\begin{rem}\label{rem-L-model}
In \cite{KM}, factorization homology of a $\oDnt$-algebra $A$ in topological spaces is defined as
$$B(D^\theta (M), D^\theta _n, A)$$
This is naturally equivalent to the definition above: the maps
$$\mathcal{D}^\theta (M) \times \mathcal{L} \to \mathcal{D}^\theta (M)$$
$$\oDntL \to \mathcal{D}^\theta _n$$
\noindent are $\Sigma$-equivariant homotopy equivalences, hence by, e.g., Section 9 of \cite{May09}, the natural transformations $\DMtL \to D^\theta (M)$, $\DntL \to D^\theta _n$ are natural weak equivalences. As in Section 2.4 of \cite{Mil}, these functors are proper, hence the associated simplicial spaces are ``good", and therefore for a $\mathcal{D}^\theta _n$-algebra $A$, the induced map
$$B(\DMtL, \DntL, A) \to B(D^\theta (M), D^\theta _n, A)$$ is a weak equivalence. 
\end{rem}

The theorem below, due to Segal, Salvatore, Lurie, Ayala-Francis, and others, identifies (for an algebra in topological spaces) factorization homology over a framed manifold as a certain mapping space. For a proof, see, for example, Theorem 2.26 of \cite{Mil}.

For a space $X$ and a based space $Y$, denote $\
\Map_c(X,Y)=\Map_*(X^+,Y)$, where $X^+$ is the one-point compactification. If $X$ is already compact, $X^+ = X_+$ is obtained by adding a disjoint basepoint.

\begin{thm}{(Non-abelian Poincar\'e duality for framed manifolds)}\label{thm-NAPD-framed}

Let $A$ be a $\mathcal{D}_n$-algebra in topological spaces. Let $M^n$ be a tangentially framed manifold. Then there is a natural scanning map
$$s: \int_M A \to \Map_c(M,B^n A)$$
\noindent which is a weak equivalence if $A$ is grouplike.
\end{thm}

For future reference, we describe this scanning map. Our model for $B^nA$ is $B(\Sigma^n_+, D_n, A)$, and the map is given on each simplicial level as
$$s_i: D(M)D_n ^i A \to \Map_c(M, \Sigma^n_+ D_n ^i A)$$
\noindent which comes from the natural transformation $S: D(M)X \to \Map_c(M, \Sigma^n_+ X)$, defined as follows: take $S(f_1,...f_n,x_1,...x_n)(m)$ to be the basepoint if $m$ is not in the image of any $f_i$, and $(f_i^{-1}(m),x_i)$ if $m \in im(f_i)$. For more details, see Section 5 of \cite{KM}. Using the natural weak equivalences in Remark \ref{rem-L-model}, we obtain a scanning map

$$B(D(M)\times L, D_n \times L, A) \to Map_c(M, B^n A)$$

\subsubsection{Axiomatic characterization}
In \cite{AF}, Ayala-Francis characterize factorization homology with coefficients in a $\oDnt$-algebra $A$, $\int_{-} A$, as a homology theory for $\theta$-framed $n$-manifolds. Objects in the $\infty$-category $Mfld_n^{\theta}$ are $\theta$-framed $n$-manifolds, and morphism spaces are $Emb^{\theta}(M,N)$ (see Example \ref{ex-operads}, and Definition 2.7 \cite{AF} or Section 2 of \cite{KM} for a definition of these embedding spaces). Disjoint union gives a symmetric monoidal product. They define a homology theory for $\theta$-framed $n$-manifolds with coefficients in a symmetric monoidal $\infty$-category $(\mathfrak{C}, \otimes)$ as a symmetric monoidal functor 
$$H: Mfld_n^{\theta} \to \mathfrak{C}$$ 
\noindent satisfying excision: if
$$M = M' \cup_{M_0 \times \mathbb{R}} M''$$
\noindent is a decomposition as $\theta$-framed manifolds, then
$$H(M')\otimes_{H(M_0 \times \mathbb{R})} H(M'') \to H(M)$$
\noindent is an equivalence. As we are working in an $\infty$-category, the relative tensor product corresponds to the derived tensor product.

The model for $\infty$-categories used in this characterization is the framework of quasicategories; in this framework, $\infty$-categories and symmetric monoidal $\infty$-categories are treated in \cite{HTT} and \cite{HA}.

Ayala-Francis prove that $\int_{-} A$ is the unique (up to weak equivalence) homology theory for $\theta$-framed $n$-manifolds satisfying
$$\int_{\mathbb{R}^n} A \simeq A$$

In Section 4 of \cite{AF}, a non-abelian Poincar\'e duality theorem for $\theta$-framed manifolds is proven using this axiomatic description. We interpret this theorem for a $\mathcal{D}_n^G$-algebra of the form $A = \Omega^n X$ for $X$ a pointed $G$-space, see Example \ref{ex-algebras}.

Define a fiber bundle $p_X$ over $BG$ with fiber $X$ by $p_X: EG \times_G X \to BG$, induced by the usual map $EG \to BG$. Given a continuous homomorphism $G \to O(n)$, let $M^n$ be a $G$-framed manifold, and let $\tau: M \to BG$ denote a lift  (given by the $G$-framing) of its tangent bundle from $BO$ to $BG$. Then we can pull $p_X$ back along $\tau$ to obtain an $X$-bundle on $M$, $\tau^* p_X$. Denote this bundle (as well as its total space) by $B^{TM}A$, and its space of compactly supported sections by $\Gamma_c(B^{TM}A)$. This gives $B^{TM}A$ as a bundle over $M$ with fiber $X$. One can think of it as being twisted by $TM$, hence it makes sense to think of $\Gamma_c(B^{TM}A)$ as a twisted mapping space.

The theorem below can be thought of as a generalization of Theorem \ref{thm-NAPD-framed}, where the mapping space is twisted by the tangent bundle of $M$. For a proof, see \cite{AF} or \cite{Sal}.

\begin{thm}{(Non-abelian Poincar\'e duality)}\label{thm-NAPD}
There is a natural equivalence
$$\int_M A \simeq \Gamma_c(B^{TM}A)$$
\end{thm}

\paragraph{Interaction with point-set model.} This axiomatic characterization agrees with Definition \ref{def-factorization-homology} (up to weak equivalence). In the axiomatic description, a $\oDnt$-algebra is a symmetric monoidal functor $Disk_n^\theta \to \mathfrak{C}$, where $Disk_n^\theta$ is the full subcategory of $Mfld_n^\theta$ on disjoint unions of $\mathbb{R}^n$. Factorization homology is the (homotopy) left Kan extension of this functor along the inclusion $Disk_n^\theta \hookrightarrow Mfld_n^\theta$. This is the derived tensor product of the right module of discs in $M$ with the algebra $A$, over the little discs operad. As in Section 7 of \cite{Hor}, we can model this as $(\DMtL) \otimes_{(\DntL)} A^c$, where $A^c$ is a cofibrant replacement of $A$ in the category of $(\DntL)$-algebras. If $A$ is cofibrant in the underlying category, we can take this cofibrant replacement to be $B(\DntL,\DntL,A)$, and the derived tensor product is thus the two-sided monadic bar construction $B(\DMtL, \DntL, A)$.

\section{Operadic behavior of Thom spectra and factorization homology}
In this section, we use the two-sided bar construction model of factorization homology, and results from Chapter 9 of \cite{LMS} on behavior of Thom spectra under monads, to describe the factorization homology of Thom spectra of $n$-fold loop maps. Chapter 9 of \cite{LMS} requires operads to be augmented over $\mathcal{L}$, so we will verify that this description still applies to Thom spectra of $n$-fold loop maps $\Omega^n X \to \Omega^n B^{n+1} O$. For $M$ framed, $\int_M Th(f)$ is a Thom spectrum of a virtual bundle over a mapping space; we will explicitly describe the map $\Map_c(M,X) \to BO$ giving this Thom spectrum. In Section 3.1, we will use this to calculate factorization homology of some Thom spectra, such as cobordism spectra and $H\mathbb{Z}/2$.

In order to show that Lewis-May Thom spectra of $E_n$-maps behave well with respect to the two-sided bar construction model of factorization homology, we will rely on a theorem of Lewis (Proposition 9.6.1, Proposition 9.6.2, and Theorem 9.7.1 of \cite{LMS}), specifying the behavior of Thom spectra under monads defined by operads that map to the linear isometries operad:

\begin{thm}{(Lewis)}\label{thm-Lewis} Let $\mathcal{C}$ be an operad augmented over $\mathcal{L}$, and let $C$ be the monad associated to $\mathcal{C}$. Let $f:X\to BO$ be a map. Then there is a natural, coherent (i.e., respecting the transformations $Id \to C$, $C^2 \to C$) isomorphism $CTh(f) \cong Th(Cf)$, where $Cf$ is defined by 
$$CX\to CBO\to BO$$
\noindent using the action of $\mathcal{C}$ on $BO$ via $\mathcal{L}$.

Furthermore, if $\mathcal{P}$ is a right $\mathcal{C}$-module, there is a natural isomorphism $PTh(f) \cong Th(Pf)$, compatible with the right $C$-module structure on $P$.

It follows that if $X$ is a $\mathcal{C}$-algebra in the category of spaces and $f$ is a $\mathcal{C}$-map, then $Th(f)$ is a $\mathcal{C}$-algebra in the category of spectra, with action induced from the isomorphism above.
\end{thm}

We use this theorem to show that the Lewis-May Thom spectrum functor ``commutes" with factorization homology; that is, to show that for $f:A \to BO$ an $E_n$-map and $M$ an $n$-manifold,

$$\int_M Th(f) \simeq Th(\int_M f)$$

The map $\int_M f$ can be described, independently of the model used for factorization homology, as the following composite:

$$\int_M A \xrightarrow{\int_M (A \to BO)} \int_M BO \xrightarrow{\int_{M \to pt} BO} BO$$

The latter map in this composite uses the fact that for an $E_\infty$-algebra, factorization homology is given by the tensor product of the manifold with the algebra.

\begin{rem}\label{rem-collapse-diag}
$\int_M BO$ is naturally equivalent to $\Omega^\infty (M_+ \wedge bo)$. This is because $\Omega^\infty(M_+ \wedge bo)$ satisfies the Ayala-Francis axioms for factorization homology: it is symmetric monoidal in $M$, its value on $\mathbb{R}^n$ is $BO$, and it satisfies excision, because $bo$-homology does. Under this equivalence, the map $\int_M BO \to BO$ is obtained by collapsing $M$ to a point, resulting in

$$\Omega^\infty(M_+ \wedge bo) \to \Omega^\infty(S^0 \wedge bo)$$

For $M$ a framed closed manifold, the scanning map $\int_M BO \to \Map(M, B^{n+1}O)$ is obtained by taking infinite loop spaces of the map of spectra

$$ M_+ \wedge bo \to F(M_+, S^n) \wedge bo$$

\noindent induced by the equivalence $S: M_+ \to F(M_+, S^n)$ which is the adjoint of

$$\xymatrix{
(M \times M)_+ \ar[r]^-{c} & M^{TM} \ar[r]^-{fr} & M_+ \wedge S^n \ar[r] & S^n
}$$

Here $c$ is the collapse map onto a tubular neighborhood of the diagonal, and $fr$ is induced by the framing of $M$. The last map collapses $M$ to a point. To see that this is equivalent to the scanning map described in Section 2, one can factor the scanning map $B(D(M),D_n,BO) \to \Map(M,B(\Sigma^n,D_n,BO))$ up to homotopy via the map

$$F: B(D(M),D_n,BO) \to \Map(M,B(D((\mathbb{R}^n)^+),D_n,BO))$$

Here $D((\mathbb{R}^n)^+)$ denotes disks in $(\mathbb{R}^n)^+$ ($=S^n$) that disappear if they touch the basepoint. The monadic bar construction $B(D((\mathbb{R}^n)^+),D_n,-)$ is a model for the $n$-fold bar construction, and gives the factorization homology of the zero-pointed manifold $(\mathbb{R}^n)^+$. For details on this more general form of factorization homology, see \cite{AF2}. The map $F$ is adjoint to a map

$$M \to \Map(B(D(M),D_n,BO),B(D((\mathbb{R}^n)^+),D_n,BO))$$

For $m \in M$, this map collapses a configuration of disks in $M$ to a small neighborhood of $m$ (disks that touch the boundary of this neighborhood vanish), and uses the framing on $M$ to canonically identify this small neighborhood of $m$ with $\mathbb{R}^n$. This then corresponds to the scanning map described above under the natural equivalences

$B(D(M),D_n,BO) \simeq \Omega^\infty(M_+ \wedge bo)$, $B(D((\mathbb{R}^n)^+),D_n,BO)) \simeq \Omega^\infty((\mathbb{R}^n)^+ \wedge bo)$

because the map described above also collapses to a small neighborhood of $m$, then uses the framing to identify this neighborhood with $\mathbb{R}^n$.

In Proposition \ref{prop-higher-hopf}, we use this description of the scanning map to provide an explicit description of the map

$$\int_{S^n} f: \Map(S^n, X) \to BO$$

\noindent for $f: \Omega^n X \to BO$ an $n$-fold loop map, where $n \in \{1,3,7\}$, so that $S^n$ is framed. (For $n=1$, this map is described in \cite{BCS}).
\end{rem}

\medskip

The theorem below describes the map $\int_M f: \int_M A \to BO$ in the monadic bar construction model of factorization homology, and gives a useful description in terms of mapping spaces if $M$ is framed. We assume that all $E_n$-spaces $A$ we consider have nondegenerate unit, and are cell complexes, so that the Thom spectrum is cofibrant in the model structure on Lewis-May spectra and the two-sided bar construction models factorization homology (see Remark \ref{rem-Thom-cellular} below), and so that the scanning map of Section 2 is a weak equivalence if $A$ is grouplike.

\begin{rem} \label{rem-Thom-cellular}
If $A$ is a $\oDntL$ algebra whose underlying space is a cell complex, and $f: A \to BO$ is a $\oDntL$-map, then $B(\DMtL,\DntL,Th(f))$ models $\int_M Th(f)$.
This is because, as in Corollary 5.5 of \cite{Blu}, the Lewis-May Thom spectrum functor for spaces over $BO$ takes cells to cells, and hence the spectrum $Th(f)$ is cellular. Cellular spectra are cofibrant in the model structure on Lewis-May spectra, so the bar construction indeed models factorization homology.
\end{rem}

\begin{thm}\label{thm-main} Let $A$ be a $\oDntL$-algebra, and let $f:A\to BO$ be a $\oDntL$-map. Let $M$ be a $\theta$-framed $n$-manifold. Then $\int_M Th(f)$ is equivalent to the Thom spectrum of the following map:
$$\xymatrix{
\int_M f: B(\DMtL,\DntL,A) \ar[r]^-{f_*} & B(\DMtL,\DntL, BO) \ar[r] & B(L,L,BO)\ar[r]^-{\sim} & BO
}$$
This description can also be applied if $f: A \to \Omega^n B^{n+1}O$ is a $\mathcal{D}_n^G$-map, as we can replace it with a $\mathcal{D}_n^G \times \mathcal{L}$-map $\tilde{A} \to BO$.

If $M^n$ is a tangentially framed manifold along with a framed embedding $i: M\times \mathbb{R}^{N-n} \hookrightarrow \mathbb{R}^N$, and $A$ is grouplike, then $\int_M Th(f)$ is the Thom spectrum of the following map: 
$$\xymatrix{
\Map_c(M, B^n A) \ar[r]^-{B^nf} & \Map_c(M,B^{n+1}O) \ar[r]^-{\sim} & \Map_c(M \times \mathbb{R}^{N-n},B^{N+1}O)\ar[r]^-{i_*} & \Map_c(\mathbb{R}^N, B^{N+1}O)
}$$
\end{thm}

Note that we can use this theorem to describe the factorization homology of Thom spectra of $n$-fold loop maps $\Omega^n X \to \Omega^n B^{n+1}O$, which comprise most of the naturally occurring examples. For a continuous homomorphism $G \to O(n)$, $\Omega^n X$ is a $\mathcal{D}_n^G$-algebra if $X$ is a pointed $G$-space. A natural $\mathcal{D}_n^G$-structure on $\Omega^n B^{n+1} O$ comes from the Salvatore-Wahl equivariant delooping. For a grouplike $\mathcal{D}_n^G$-algebra $T$, Salvatore-Wahl construct in Theorem 3.1 of \cite{SW} a $G$-space $B^n T$ such that $\Omega^n B^n T$ is connected to $A$ by a zigzag of weak equivalences of $\mathcal{D}_n^G$-algebras. As a space, $B^n T$ is given as by $B(\Sigma^n, \mathbb{D}_n, T)$ (or $B(\Sigma^n_+, D_n, T)$), and the $G$-action is obtained by realizing the simplicial $G$-space $\Sigma^n \mathbb{D}_n^{\bullet} T$. The $G$-action on $T$ is the given one, the action on $\mathbb{D}_n$ is by rotation and reflection of discs via $O(n)$, and the action on $\Sigma^n$ is as the 1-point compactification of $\mathbb{R}^n$, via the action of $O(n)$. In the case $T=BO$, we use the trivial $G$-action on $BO$.

\begin{proof}
By Lewis's theorem above, there are a natural, coherent isomorphisms
$$(\DntL) Th(f) \cong Th((\DntL) f)$$
$$(\DMtL) Th(f) \cong Th((\DMtL) f)$$
\noindent as both symmetric sequences are augmented over the linear isometries operad. (The maps $(\DntL) f$, $(\DMtL) f$ are defined as in Theorem \ref{thm-Lewis} by the augmentation over $\mathcal{L}$ and the $\mathcal{L}$-action on $BO$). Therefore we get an isomorphism of simplicial objects 
$$(\DMtL) (\DntL) ^p Th(f) \cong Th((\DMtL) (\DntL) ^p f)$$
The Thom spectrum functor commutes with colimits and tensors over unbased spaces, and therefore with geometric realization. This gives the bar construction description description of the map whose Thom spectrum is $\int_M Th(f)$.

\medskip

We now address the case of $f:A \to \Omega^n B^{n+1} O$ a $\mathcal{D}_n^G$-map (for example, a suitable $n$-fold loop map.) Note that there is a zigzag of weak equivalences of $\mathcal{D}_n^G \times \mathcal{L}$-algebras between $\Omega^n B^{n+1} O$ and $BO$, given by the scanning map:

$$\xymatrix{
\Omega^n B(\Sigma^n_+, D_n \times L, BO) & \ar[l]_-{\sim} B(D_n \times L, D_n \times L, BO) \ar[r]^-{\sim} & BO
}$$

By Lemma 2.3 of \cite{SS}, the category of algebras over $\mathcal{D}_n^G \times \mathcal{L}$ carries a cofibrantly generated model structure induced from that on spaces by the free-forgetful adjunction, hence every object is fibrant and weak equivalences are determined in the underlying category of topological spaces. Thus we can functorially cofibrantly replace $\mathcal{D}_n^G \times \mathcal{L}$-algebras, resulting in cofibrant, fibrant algebras, between which we can invert weak equivalences. This results in a $\mathcal{D}_n^G \times \mathcal{L}$-map $A^c \to BO$, where $A^c$ is the functorial cofibrant replacement of $A$ in the category of $\mathcal{D}_n^G \times \mathcal{L}$-algebras.

\medskip

Now suppose $M$ is tangentially framed, with a framed embedding as above, and $A$ is grouplike. Then non-abelian Poincar\'e duality (Theorem \ref{thm-NAPD-framed}) holds. It is natural in the algebra and manifold variables, which yields the required description of the map. More explicitly, consider the commutative diagram
$$\xymatrix{
B(D(M)\times L,D_n \times L,A) \ar[r]^-{\sim} \ar[d]  & \Map_c(M, B(\Sigma^n_+,D_n \times L,A)) \ar[d] \\
B(D(M)\times L,D_n \times L,BO) \ar[r]^-{\sim} \ar[d]  & \Map_c(M, B(\Sigma^n_+,D_n \times L,BO)) \ar[d] \\
B(D(M\times \mathbb{R}^{N-n})\times L,D_N \times L,BO) \ar[r]^-{\sim} \ar[d]  & \Map_c(M \times \mathbb{R}^{N-n}, B(\Sigma^N_+,D_N \times L,BO)) \ar[d] \\  
B(D(\mathbb{R}^N)\times L,D_N \times L,BO) \ar[r]^-{\sim} \ar[d]^-{\wr}  & \Map_c(\mathbb{R}^N, B(\Sigma^N_+,D_N \times L,BO)) \simeq BO \\
B(L,L,BO) \ar[d]^-{\wr}\\
BO                           
}
$$

The horizontal maps are scanning maps; we use $B(\Sigma^n_+,D_n \times L,A)$ as an $n$-fold delooping of $A$. The Lewis-May Thom spectrum functor takes weak equivalences over $BO$ to weak equivalences, thus $\int _M Th(f)$ is equivalent to the Thom spectrum of the composition of the right-hand column, and we can conclude.
\end{proof}

\begin{rem}\label{rem-P-scanning}
Remembering that $\Map_c(M \times \mathbb{R}^k,-)=\Map_*(\Sigma^k(M^+),-)$, where $M^+$ denotes the one-point compactification of $M$, we can alternatively describe the map in the Theorem \ref{thm-main} as the composite:
$$\xymatrix{
\Map_*(M^+,B^n A) \ar[d]^{B^nf} \\
\Map_*(M^+,B^{n+1}O) \ar[d]^{\Sigma^{N-n}} \\
\Map_*(\Sigma^{N-n}(M^+),\Sigma^{N-n}B^{n+1}O) \ar[d]^-\wr \\
\Map_*(\Sigma^{N-n}(M^+),B^{N+1}O) \ar[d]^{c^*} \\
\Map_*(S^N,B^{N+1}O) \simeq BO
}$$

The bottom-most vertical map above is induced by the Pontryagin-Thom collapse map $S^N \to \Sigma^{N-n}(M^+)$.

This description can also be obtained from the model independent description of $\int_M f$, using the fact that there is an essentially unique embedding of $M$ into $\mathbb{R}^\infty$.
\end{rem}

\begin{rem}\label{rem-BO}
Theorem \ref{thm-main} is stated, for simplicity, for Thom spectra of maps to $BO$. The same is true for Thom spectra of maps $f: A \to BF$, where $BF$ is the classifying space for stable spherical fibrations. In this case, instead of the Lewis-May Thom spectrum as stated, one first replaces $f$ by a fibration $\Gamma f$ and uses $Th(\Gamma f)$, see Chapter 9.3 of \cite{LMS}. $\Gamma$ behaves well with respect to operadic structures, and one can show that $Th(\Gamma f)$ takes a cell complex $A$ to a spectrum which is homotopy equivalent to a cellular spectrum. More generally, notice that $BF$ is $BGL_1(S)$, where $S$ is the sphere spectrum; a similar description holds for generalized Thom spectra of maps to $BGL_1(R)$, where $R$ is a commutative ring spectrum. This is most easily obtained using the $\infty$-categorical approach to Thom spectra of \cite{ABGHR}. See Section 4 for results on generalized Thom spectra.
\end{rem}

We end this subsection by explicitly describing $\int_{S^n} f$ for $n \in \{1,3,7\}$. This is a direct generalization of the description of topological Hochschild homology of a Thom spectrum in Theorem 1 of \cite{BCS}, which deals with the case $n=1$. Recall that an element $\alpha \in \pi_n^s$ gives a map $B^{n+1}O \to BO$ by considering $\alpha$ as a stable map $S^n \to S^0$, and forming

$$\Omega^\infty \alpha: \Omega^\infty(S^n \wedge bo) \to \Omega^\infty(S^0 \wedge bo)$$

\begin{prop}\label{prop-higher-hopf}
If $n \in \{1,3,7\}$ (so that $S^n$ is framed), $X$ is $(n-1)$-connected, and $f: \Omega^n X \to BO$ is an $n$-fold loop map, then $\int_{S^n} Th(f)$ is equivalent to the Thom spectrum of $l^n(f): \Map(S^n, X) \to BO$, defined by

$$\xymatrix{
\Map(S^n,X) \ar[r]^-{f_*} & \Map(S^n, B^{n+1}O) \ar[r]^-{\phi^{-1}}_-{\sim} & BO \times B^{n+1} O \ar[r]^-{id \times -\eta_n} & BO \times BO \ar[r]^-{mult} & BO
}$$

Here $\eta_n$ denotes the Hopf map in $\pi_n^s$; that is, $\eta_1 = \eta$, $\eta_3 = \nu$ and $\eta_7 = \sigma$. We use $\phi^{-1}$ to denote a homotopy inverse to the equivalence $\phi: B^{n+1}O \times \Omega^n B^{n+1}O \to \Map(S^n,B^{n+1}O)$, which is induced by the inclusion of constant loops $B^{n+1}O \to \Map(S^n,B^{n+1}O)$ and the inclusion of $\Omega^n B^{n+1}O$ into the mapping space.
\end{prop}

\begin{proof}
We proceed as in the proof of Proposition 7.3 of \cite{Sch04}, using the description of the scanning map in Remark \ref{rem-collapse-diag} in place of the standard equivalence between the cyclic bar construction and the free loop space. We have the following commutative diagram of spectra, induced by the (co)fiber sequence $S^0 \to S^n_+ \to S^n$, which has a stable splitting induced by $r: S^n_+ \to S^0$:

$$\xymatrix{
S^0 \ar[r] \ar[d]^-{\wr} & S^n _+ \ar[r] \ar[d]^-{S} & S^n \ar[d]^-{=}\\
F(S^n,S^n) \ar[r] & F(S^n_+, S^n) \ar[r] & F(S^0,S^n)
}$$

The second row is obtained by mapping the first row into $S^n$. $S$ is the scanning map from Remark \ref{rem-collapse-diag}, adjoint to 

$$(S^n \times S^n)_+ \to S^{TS^n} \cong (S^n_+) \wedge S^n \to S^n$$

\noindent in which the last map is induced by $r: S^n_+ \to S^0$. To be more explicit, we use the structure of an H-space with inverses on $S^n$ for $n \in \{1,3,7\}$ and define $C: (S^n \times S^n)_+ \to S^n$, the adjoint of $S$, to send $(x,y) \in S^n \times S^n$ to $x^{-1}y$ if $x,y$ are not antipodal, and the basepoint otherwise; this is adjoint to the scanning map because it takes $(x,y)$ to $x^{-1}y$ (which is in a disk neighborhood of the identity) if $x$ and $y$ are close (that is, not antipodal to each other), and to the basepoint otherwise. We take the basepoint of the source $S^n$ in the second row to be the H-space identity element $1$, and the basepoint of the target $S^n$ in the second row to be its antipode $-1$, so that $S$ is continuous. We take the basepoint of the first row $S^n$ to be $-1$, so that the horizontal map $S^n_+ \to S^n$ sends the disjoint basepoint to $-1$. The left-hand vertical map sends the non-basepoint of $S^0$ to $-id: S^n \to S^n$ (which is a pointed map due to our choice of basepoints). The right hand vertical map is adjoint to the map which sends the non-basepoint of $S^0$ to the map $x \mapsto x^{-1}$. With these choices, one can check that the diagram commutes.

\medskip

The map $\int_{S^n} id: \int_{S^n} BO \to BO$ is, under the natural equivalence $\int_{S^n} BO \simeq \Omega^\infty(S^n_+ \wedge bo)$, given by

$$\Omega^\infty(r \wedge bo): \Omega^\infty(S^n_+ \wedge bo) \to \Omega^\infty(bo)$$

We would therefore like to determine the homotopy class of the map

$$\xymatrix{
B^{n+1}O \ar[r]^-{const} & \Map(S^n, B^{n+1} O) \ar[r]^-{S^{-1}} & \Omega^\infty (S^n_+ \wedge bo) \ar[r]^-{r \wedge bo} & BO
}$$

This is obtained by smashing $bo$ with the composition

$$\xymatrix{
F(S^0,S^n) \ar[r]^-{r^*} & F(S^n_+, S^n) \ar[r]^-{S^{-1}} & S^n _+ \ar[r]^-{r} & S^0
}$$

Thus, we aim to show that this composition is homotopic to $-\eta_n$. As in the proof of Proposition 7.3 of \cite{Sch04}, we represent $S$ as a $2 \times 2$ matrix with respect to the splittings

$$S^n_+ \simeq S^0 \vee S^n$$

\noindent and

$$F(S^n_+,S^n) \simeq F(S^n,S^n) \times F(S^0,S^n)$$

\noindent induced by $r$ and $r^*$ respectively. The map we would like to determine, $r \circ S^{-1} \circ r^*$, is the off-diagonal term in the matrix representing $S^{-1}$, thus is the negative of the off-diagonal term in the matrix representing $S$. We will denote this off-diagonal term by $S_{12}$. Denote the stable section $S^n \to S^n_+$ by $\mathbf{s}$; then $S_{12}$ is homotopic to

$$\xymatrix{
S^n \ar[r]^-{\mathbf{s}} & S^n _+ \ar[r]^-{S} & F(S^n_+, S^n) \ar[r]^-{\mathbf{s}^*} & F(S^n, S^n)
}$$

This is adjoint to

$$\xymatrix{
S^n \wedge S^n \ar[r]^-{\mathbf{s} \wedge \mathbf{s}} & (S^n \times S^n)_+ \ar[r]^-{C} & S^n_+ \ar[r] & S^n
}$$

Here $C(x,y) = x^{-1}y$ is the adjoint of $S$. The last map collapses the disjoint basepoint to the basepoint of $S^n$, which we have taken to be $-1$. This stable map $S^n \wedge S^n \to S^n$ represents the Hopf map, and therefore $-S_{12} = -\eta_n$, and this gives the required map $B^{n+1} O \to BO$ in the description. Thus $\int_{S^n} f$ is homotopic to the map $l^n(f)$ described in the proposition, as required.
\end{proof}

\subsection{Calculations}
We now use Theorem \ref{thm-main} to compute some examples of factorization homology. The proposition below addresses cases in which the map $f: A \to BO$ is more highly commutative.

Recall that for a stably framed manifold $M^n$ properly embedded along with a tubular neighborhood in $\mathbb{R}^N$, $M \times \mathbb{R}^{N-n} \hookrightarrow \mathbb{R}^N$, we have the Pontryagin-Thom collapse map
$$ c: S^N = (\mathbb{R}^N)^+ \to (M \times \mathbb{R}^{N-n})^+ \cong \Sigma^{N-n}(M^+)$$

Notice that the composition

$$\xymatrix{
S^N \ar[r]^-{c} & \Sigma^{N-n}(M^+) \ar[r]^-{deg} & S^N
}$$

\noindent where $deg$ collapses all but a small disk in $M$ to a point, is homotopic to the identity. The map $deg$ is part of a cofiber sequence

$$\xymatrix{
\Sigma^{N-n}(M^+ -D^n) \ar[r]^-{inc} & \Sigma^{N-n}(M^+) \ar[r]^-{deg} & S^N
}$$

That is, the Pontryagin-Thom collapse map stably splits this cofiber sequence. We would like to use the stable splitting of this cofiber sequence to provide a relatively simple description of $\int_{M \times \mathbb{R}^{N-n}} Th(f)$. In order to use Theorem \ref{thm-main}, we require $f$ to be an $E_N$-map.

\begin{prop}\label{prop-E_N-map}
Let $M$ be a connected, stably framed $n$-manifold along with an embedding $M \times \mathbb{R}^{N-n} \hookrightarrow \mathbb{R}^N$, and let $f:A \to BO$ be an $E_N$-map, where $A$ is grouplike. Then
$$ \int_{M \times \mathbb{R}^{N-n}} Th(f) \simeq Th(f) \wedge \Map_* (M^+ - D^n, B^n A)_+$$
\end{prop}

If $M$ is not stably framed, the situation is more complicated: $M \times \mathbb{R}^{N-n}$ does not embed into $\mathbb{R}^N$, and non-abelian Poincar\'e duality gives a section space rather than a mapping space.

\begin{proof}
By Theorem \ref{thm-main}, $\int_{M \times \mathbb{R}^{N-n}} Th(f)$ is equivalent to the Thom spectrum of the following map:

$$\xymatrix{
\Map_*(M^+,B^n A) \ar[d]^-{\wr} \\
\Map_*(\Sigma^{N-n}(M^+),B^N A) \ar[d]^-{B^N f} \\
\Map_*(\Sigma^{N-n}(M^+),B^{N+1}O) \ar[d]^-{c^*} \\
\Map_*(S^N,B^{N+1}O) \simeq BO
}$$

This composite is the top row of the diagram below, which commutes because $\Map_*(-,-)$ is a bifunctor:

$$\xymatrix{
\Map_*(\Sigma^{N-n}(M^+),B^N A)  \ar[r] \ar[rd]^-{c^*} & \Map_*(S^N,B^{N+1}O) \simeq BO \\
 & \Map_*(S^N,B^N A)\simeq A \ar[u]^{f_*} 
}$$

Hence the Thom spectrum of 
$$ \Map_*(\Sigma^{N-n}(M^+),B^N A) \to \Map_*(S^N,B^N A) \to \Map_*(S^N,B^{N+1}O) \simeq BO$$
\noindent is also equivalent to $\int_{M \times \mathbb{R}^{N-n}} Th(f)$.

The map 
$$\Map_*(\Sigma^{N-n}(M^+),B^N A) \to \Map_*(S^N,B^N A) \times \Map_*(\Sigma^{N-n}(M^+ - D^n), B^N A)$$
\noindent induced by the Pontryagin-Thom collapse map and by the inclusion in the cofiber sequence is an equivalence; this follows from the stable splitting of the degree cofiber sequence. Thus we have the following commutative diagram

$$\xymatrix{
\Map_*(\Sigma^{N-n}(M^+),B^N A)  \ar[r] \ar[rd]^-{c^*} \ar[d]^-{\wr} & \Map_*(S^N,B^{N+1}O) \simeq BO \\
\Map_*(S^N,B^N A) \times \Map_*(\Sigma^{N-n}(M^+ - D^n), B^N A) \ar[r] & \Map_*(S^N,B^N A)\simeq A \ar[u]^{f_*} 
}$$

The bottom horizontal arrow is given by projection onto the first factor. Thus $\int_{M \times \mathbb{R}^{N-n}} Th(f)$ is equivalent to the Thom spectrum of 
$$\Map_*(S^N,B^N A) \times \Map_*(\Sigma^{N-n}(M^+ - D^n), B^N A) \to \Map_*(S^N,B^N A) \to \Map_*(S^N, B^{N+1}O) \simeq BO$$
\noindent which is equivalent to $Th(f) \wedge \Map_*(\Sigma^{N-n}(M^+ - D^n), B^N A)_+$. Notice that
$$\Map_*(\Sigma^{N-n}(M^+ - D^n), B^N A) \simeq \Map_* (M^+ - D^n, B^n A)$$
\noindent and the conclusion follows.
\end{proof}

When $f: A \to BO$ is an infinite loop map this recovers Theorem 1 of \cite{Sch}; 
factorization homology is known to agree with higher (topological) Hochschild homology for $E_\infty$-algebras. See Theorem 5 of \cite{GTZ14} for the CDGA case, or Proposition 5.1 of \cite{AF} in general.

\begin{rem}\label{rem-framed-alg}
If $A$ is an $E_N^{SO(N)}$-algebra, then $\int_M A \simeq \int_{M \times \mathbb{R}^{N-n}} A$. In this case, the description in Proposition \ref{prop-E_N-map} also applies to $\int_M Th(f)$. In particular, this is true for any $E_{\infty}$-map.
\end{rem}

This result allows us to explicitly describe $\int_M E$ for $E$ any cobordism ring spectrum.

\begin{cor} \label{cor-lie-gps} If $M$ is a connected, stably framed $n$-manifold and $G$ is a stabilized Lie group (e.g., $O,SO,Spin,U,Sp$), then
$$\pi _* (\int_M MG) \cong \Omega_* ^G (\Map_*(M^+ - D^n, B^{n+1} G))$$
\end{cor}

This result recovers the computations of \cite{BCS} and \cite{Sch} in the case $M=S^n$.

\begin{proof}
In this case, $MG$ is the Thom spectrum of the $\mathcal{L}$-map $BG \to BO$, that is, a $\mathcal{D}_N ^\theta \times \mathcal{L}$-map for all $N$. Thus the result follows from Proposition \ref{prop-E_N-map}.
\end{proof}

Next, we consider Thom spectra that arise from systems of groups {$G_n$} with a block sum pairing and compatible homomorphisms $G_n \to O(n)$. We will consider the examples $\Sigma_n \to O(n)$ and $GL_n(\mathbb{Z}) \to O(n)$, with Thom spectra $M\Sigma$ and $MGL(\mathbb{Z})$, respectively. These are cobordism spectra for manifolds whose stable normal bundle has a specific flat connection; that is, the structure group of the stable normal bundle reduces to the symmetric group or the general linear group of the integers, respectively. For these Thom spectra, Proposition \ref{prop-E_N-map} gives the following corollary:

\begin{cor} \label{cor-sys-gps} Let $M$ be a connected, stably framed $n$-manifold. Then
$$\int_M M\Sigma \simeq M\Sigma \wedge \Map_*(M^+ - D^n, (QS^n) \langle n \rangle)_+$$
$$\int_M MGL(\mathbb{Z}) \simeq MGL(\mathbb{Z}) \wedge \Map_*(M^+ - D^n, B^n (BGL(\mathbb{Z}))^+)_+$$
Here the $(-)^+$ in $(BG)^+$ denotes plus construction, and $(-) \langle n \rangle$  denotes $n$-connected cover. 
\end{cor}
\begin{proof}
The maps $B\Sigma \to (B\Sigma)^+$, $BGL(\mathbb{Z}) \to BGL(\mathbb{Z})^+$ are homology equivalences, thus induce equivalences on mod 2 homology. By the universal property of the plus construction, $B\Sigma \to BO$ and $BGL(\mathbb{Z}) \to BO$ factor through maps

$$f_\Sigma: (B\Sigma)^+ \to BO$$
$$f_{GL}: BGL(\mathbb{Z})^+ \to BO$$

The classifying space of the stabilized braid group $BBr$ maps to $B\Sigma$ and $BGL(\mathbb{Z})$ over $BO$, and according to Lemma 2.2 of \cite{CohF}, $MBr \simeq H\mathbb{Z}/2$. Hence $M\Sigma$ and $MGL(\mathbb{Z})$ are $H\mathbb{Z}/2$-module spectra. They are equivalent to $Th(f_\Sigma)$, $Th(f_{GL})$, which are also $H\mathbb{Z}/2$-modules, after smashing with $H\mathbb{Z}/2$, therefore

$$M\Sigma \simeq Th(f_\Sigma)$$
$$MGL(\mathbb{Z}) \simeq Th(f_{GL})$$

The maps $f_\Sigma$, $f_{GL}$ are $E_\infty$-maps; one way to see this is to observe that they are the 0-components of the maps

$$\Omega B(\coprod_n B\Sigma_n \to \coprod_n BO(n))$$
$$\Omega B(\coprod_n BGL_n(\mathbb{Z}) \to \coprod_n BO(n))$$

\noindent which are obtained by group-completing maps of $E_\infty$-spaces (where the $E_\infty$-structure comes from the block sum pairing). Thus $M\Sigma$, $MGL(\mathbb{Z})$ are $E_{\infty}$-ring spectra. Proposition \ref{prop-E_N-map} gives the required description of factorization homology of Thom spectra obtained from infinite loop maps to $BO$, and the corollary follows. In the description of $\int_M M\Sigma$, $(QS^n) \langle n \rangle = (\Omega^{\infty} \Sigma^{\infty} S^n) \langle n \rangle$ appears because, due to the Barratt-Priddy-Quillen theorem \cite{BP}, $(B\Sigma)^+ \simeq Q_0S^0$.
\end{proof}

\subsubsection{Factorization homology of $H\mathbb{Z}/2$.}
By a theorem of Mahowald, $H\mathbb{Z}/2$ is equivalent to the Thom spectrum of a 2-fold loop map $\gamma: \Omega^2 S^3 \to BO$. (For a proof, see \cite{Pri}.) Furthermore, this is an equivalence of $E_2$-ring spectra. We will use this to calculate $\int_M H\mathbb{Z}/2$ for oriented surfaces $M$.

Let $\alpha: S^1 \to BO$ represent the generator of $\pi_1(BO)$. It is known (see, e.g., \cite{Coh}) that $H\mathbb{Z}/2 \simeq D_2(Th(\alpha))$ as an $E_2$-ring spectrum. In fact, $H\mathbb{Z}/2$ is equivalent as an $E_2$-ring spectrum to the Thom spectrum of the map $\gamma$, given by

$$\xymatrix{
(D_2 \times L)S^1 \ar[r]^-{(D_2 \times L)\alpha} & (D_2 \times L)BO \ar[r]^-{act} & BO \\
}$$

In order to use this to calculate factorization homology over orientable (not just framed) surfaces, we will show below that this is furthermore an equivalence of $E_2^{SO(2)}$-algebras. Recall that, via equivalences $(D_2 \times L)S^1 \simeq \Omega^2 S^3$ and $BO \simeq \Omega^2 B^3 O$, Mahowald's map is homotopic to the twice looping of a generator of $\pi_3 (B^3 O)$.

Mahowald's map $\gamma$ above is a $\mathcal{D}_2^{G} \times \mathcal{L}$-map for any $G \to O(2)$, if we take the Salvatore-Wahl rotation action on $D_2$. (The group action on $BO$ is trivial, because the operad action is via the projection to $\mathcal{L}$.) This makes $Th(\gamma)$ a $(\mathcal{D}_2^G \times \mathcal{L})$-algebra. The equivalence 

$$(D_2 \times L)S^1 \to D_2 S^1 \to \Omega^2 S^3$$

\noindent is then an equivalence of $\mathcal{D}_2^{G}$-algebras, provided the group action on $S^3$ is as $S^3 = (\mathbb{R}^2)^+ \wedge S^1$, with $G$ acting via the homomorphism $G \to O(2)$ on $\mathbb{R}^2$ and trivially on the $S^1$ smash factor.

\begin{lem} \label{lem-equiv-E_2-fr} For any continuous homomorphism $G \to O(2)$, $H\mathbb{Z}/2$ is equivalent as a $\mathcal{D}_2 ^G \times \mathcal{L}$-algebra to the Thom spectrum of the map
$$\gamma: (D_2 \times L)S^1 \to (D_2 \times L) BO \to BO$$
The $\mathcal{D}_2 ^G \times \mathcal{L}$-algebra structure on $H\mathbb{Z}/2$ is obtained via the augmentation to $\mathcal{L}$. 
\end{lem}
\begin{proof}
By Mahowald's theorem, $H\mathbb{Z}/2 \simeq Th(\gamma)$. This equivalence is obtained via the following sequence of maps, the first one being the natural isomorphism in Lewis's theorem:
$$Th(\gamma) \cong (D_2 \times L)Th(S^1 \to BO) \to (D_2 \times L)MO \to (D_2 \times L)H\mathbb{Z}/2 \to H\mathbb{Z}/2$$
The last map is given by the action of $D_2 \times L$ on $H\mathbb{Z}/2$, which comes from the $\mathcal{L}$-algebra structure on this Eilenberg-MacLane spectrum. The second-to-last map is given by the Thom class.

For any $X$, $(D_2 \times L) X$ has a $\mathcal{D}_2 ^G \times \mathcal{L}$-algebra structure via the ``rotation of discs" $O(2)$-action on $D_2$. This $\mathcal{D}_2 ^G \times \mathcal{L}$ action on $(D_2 \times L)Th(S^1 \to BO)$ is compatible with the one on $Th(\gamma)$, due to the fact that $\gamma$ is a $\mathcal{D}_2 ^G \times \mathcal{L}$-map. The equivalence $Th(\gamma) \to H\mathbb{Z}/2$ above is a $\mathcal{D}_2 ^G \times \mathcal{L}$-map; this is clear for all but the last map, which is  a $\mathcal{D}_2 ^G \times \mathcal{L}$-map because the action of $\mathcal{D}_2 \times \mathcal{L}$ on $H\mathbb{Z}/2$ comes from the projection to $\mathcal{L}$, so the relevant diagram commutes. Hence we can conclude.
\end{proof}

For example, for $G= SO(2)$, this will help us compute factorization homology of $H\mathbb{Z}/2$ over orientable surfaces. We will use the following proposition:

\begin{lem}\label{lem-fact-hom-mahowald}
For $\Omega^2 S^3$ with $E_2^{SO(2)}$-structure as above, and $M$ a closed or punctured genus $g$ surface,
$$\int_M \Omega^2 S^3 \simeq \Map_c(M,S^3)$$
\end{lem}

\begin{proof}
This is automatically true for parallelizable $M$ (hence for punctured genus $g$ surfaces) by non-abelian Poincar\'e duality. For closed orientable surfaces, we need to show that the section space given by non-abelian Poincar\'e duality is in fact equivalent to a mapping space.

According to non-abelian Poincar\'e duality as in Section 4 of \cite{AF}, $\int_M \Omega^2 S^3$ is equivalent to the space of compactly supported sections of the bundle $B^{TM}\Omega^2 S^3$, obtained as follows (see also Section 2):

Consider the bundle $\sigma_3$ over $BSO(2)$ with total space $ESO(2) \times_{SO(2)} S^3$, with $SO(2)$-action on $S^3$ as described above; that is, $SO(2)$ acts as a matrix group on $S^2 = (\mathbb{R}^2)^+$, and acts trivially on the $S^1$ factor in $S^3 = S^2 \wedge S^1$. Denote by $B^{TM}\Omega^2 S^3$ the pullback $\tau^* \sigma_3$ of this bundle to $M$, along the classifying map $\tau : M \to BSO(2)$ of $TM$. Notice that

$$\sigma_3 \cong \Sigma_{BSO(2)}(ESO(2) \times_{SO(2)} S^2)$$ 

\noindent where $\Sigma_{BSO(2)}$ denotes fiberwise suspension, and $ESO(2) \times_{SO(2)} S^2$ is the fiberwise 1-point compactification of the tautological plane bundle over $BSO(2)$. Thus $\tau^* \sigma_3 \cong \Sigma_M T^+ M$. (Here $T^+M$ denotes the fiberwise 1-point compactification of $TM$). Therefore

$$\tau^* \sigma_3 \cong \Sigma_M T^+ M \cong T^+(M \oplus \mathbb{R})$$

For $M$ an oriented surface, this bundle is trivial, and the section space is a mapping space.
\end{proof}

\begin{prop}\label{prop-HZ/2}
Let $M$ be a genus $g$ surface or a punctured genus $g$ surface. Then
$$\int_M H\mathbb{Z}/2 \simeq H\mathbb{Z}/2 \wedge \Map_*(M^+ - D^2,S^3)_+$$
In particular,
$$\int_{\Sigma_g} H\mathbb{Z}/2 \simeq H\mathbb{Z}/2 \wedge (S^3 \times (\Omega S^3)^{2g})_+$$
\end{prop}

The homology of $S^3$ and $\Omega S^3$ can easily be described, so we obtain the higher and iterated $THH$ of $H\mathbb{Z}/2$:

\begin{cor}\label{cor-higher-THH}
$$THH^{S^2}(H\mathbb{Z}/2) \cong \Lambda_{\mathbb{Z}/2}[x]$$
$$THH^{\mathbb{T}^2}(H\mathbb{Z}/2) \cong \Lambda_{\mathbb{Z}/2}[x]\otimes_{\mathbb{Z}/2} \mathbb{Z}/2[y_1, y_2]$$

\noindent where $|x|=3, |y_i|=2$.
\end{cor}

\begin{proof}
By Lemma \ref{lem-equiv-E_2-fr}, $H\mathbb{Z}/2$ is equivalent as an $E_2^{SO(2)}$-algebra to the Thom spectrum of an $E_2^{SO(2)}$-map $\Omega^2 S^3 \to BO$. Thus, by Theorem \ref{thm-main}, $\int_M H\mathbb{Z}/2$ is equivalent to the Thom spectrum of a map $\int_M \Omega^2 S^3 \to BO$. By Lemma \ref{lem-fact-hom-mahowald}, $\int_M \Omega^2 S^3 \simeq \Map_c(M,S^3)$, and thus by a spectrum-level version of the Thom isomorphism (see, e.g.,  Theorem 9.5.6 of \cite{LMS}),

$$H\mathbb{Z}/2 \wedge \int_M H\mathbb{Z}/2 \simeq H\mathbb{Z}/2 \wedge \Map_c(M,S^3)_+$$

$S^3$ is a group, so we can deloop it to get
$$\Map_c(M,S^3) = \Map_*(M^+,S^3) \simeq \Map_*(\Sigma(M^+),BS^3)$$

The space $\Sigma(\Sigma_g)$ splits as $\Sigma S^2 \vee \bigvee_{2g} \Sigma S^1$, because the attaching map of the 2-cell of $\Sigma_g$ is nullhomotopic. Thus the cofiber sequence

$$\Sigma((\Sigma_g )_+ - D^2) \to \Sigma( (\Sigma_g)_+) \to S^3$$

\noindent splits, and therefore

$$\Map_*(\Sigma(M^+),BS^3) \simeq \Map_*(S^3,BS^3) \times \Map_*(\Sigma(M^+ - D^2), BS^3)$$
 
Thus we have $\Map_c(M,S^3) \simeq \Omega^2 S^3 \times \Map_*(M^+ -D^2, S^3)$, and hence
$$H\mathbb{Z}/2 \wedge \Map_c(M,S^3)_+ \simeq H\mathbb{Z}/2 \wedge (\Omega^2 S^3)_+ \wedge \Map_*(M^+ -D^2, S^3)_+$$

By Mahowald's theorem, $H\mathbb{Z}/2$ is the Thom spectrum of a map $\Omega^2 S^3 \to BO$, thus by the Thom isomorphism

$$H\mathbb{Z}/2 \wedge (\Omega^2 S^3)_+ \simeq H\mathbb{Z}/2 \wedge H\mathbb{Z}/2$$

Following this chain of equivalences, we get

$$H\mathbb{Z}/2 \wedge \int_M H\mathbb{Z}/2 \simeq H\mathbb{Z}/2 \wedge H\mathbb{Z}/2 \wedge \Map_*(M^+ -D^2, S^3)_+$$

The spectra $\int_M H\mathbb{Z}/2$ and $H\mathbb{Z}/2 \wedge \Map_*(M^+ -D^2, S^3)_+$ are $H\mathbb{Z}/2$-module spectra; the latter in the obvious way, and the former is in fact an $H\mathbb{Z}/2$-algebra via the map

$$\int_{\mathbb{R}^2} H\mathbb{Z}/2 \to \int_M H\mathbb{Z}/2$$

\noindent obtained from an embedding of a small disc into $M$. In order to obtain this $H\mathbb{Z}/2$-algebra structure, one uses the fact that factorization homology of an $E_\infty$-algebra is naturally an $E_\infty$-algebra. Because these are $H\mathbb{Z}/2$-modules, they are generalized Eilenberg-MacLane spectra, of the form $\bigvee_i \Sigma^{k_i} H\mathbb{Z}/2$. Their homotopy type is thus completely determined by the degrees in which there are copies of $H\mathbb{Z}/2$, and the number of copies in each degree. Thus it suffices to check that these spectra are equivalent after smashing with $H\mathbb{Z}/2$, and

$$\int_M H\mathbb{Z}/2 \simeq H\mathbb{Z}/2 \wedge \Map_*(M^+ - D^2,S^3)_+$$
\noindent as required.
\end{proof}

For a slightly different way of using Thom isomorphism to obtain an equivalence $\int_M H\mathbb{Z}/2 \simeq H\mathbb{Z}/2 \wedge \Map_*(M^+ - D^2,S^3)_+$, see the proof of Proposition \ref{prop-HZ(p)}.

\medskip

In the next section, we will use a $p$-local Thom spectrum functor to prove a similar result for $H\mathbb{Z}/p$, $H\mathbb{Z}_{(p)}$, and $H\mathbb{Z}$.

\section{Generalized Thom spectra}

In this section, we will use the $\infty$-categorical approach to generalized Thom spectra of \cite{ABGHR} to describe the factorization homology of Thom spectra $Th(f)$, for $f: A \to BGL_1(R)$ an $E_n$-map, where $R$ is a commutative ring spectrum. We restrict to commutative $R$ so that $BGL_1(R)$ is an $E_\infty$-space, and the generalized Thom spectrum functor is a symmetric monoidal functor.

Ordinary Thom spectra fit into this framework by taking $R$ to be the sphere spectrum $S$. Symmetric monoidal properties of the generalized Thom spectrum functor, as in \cite{ABG}, will allow us to describe the factorization homology of these Thom spectra as we described it for ordinary Thom spectra in Section 3. We will then rely on a description of some Eilenberg-MacLane spectra as generalized Thom spectra to calculate the factorization homology of $H\mathbb{Z}/p$, $H\mathbb{Z}_{(p)}$, and $H\mathbb{Z}$ over oriented surfaces.

By an observation due to Hopkins (see \cite{MRS}), for an odd prime $p$, $H\mathbb{Z}/p$ is the generalized Thom spectrum of a 2-fold loop map $\gamma_p: \Omega^2 S^3 \to BGL_1(S_{(p)})$, where $S_{(p)}$ denotes the $p$-local sphere spectrum. As in the $p=2$ case, this map can be obtained by extending the map $S^1 \to BGL_1(S_{(p)})$ given by the unit $(1-p)$ to a 2-fold loop map:

$$(D_2 \times L)S^1 \to (D_2 \times L)BGL_1(S_{(p)}) \to BGL_1(S_{(p)})$$

We will express $\int_M H\mathbb{Z}/p$ as a Thom spectrum of a map $\int_M \Omega^2 S^3 \to BGL_1(S_{(p)})$, which will allow us to use the arguments of the previous section to calculate $\int_M H\mathbb{Z}/p$ for orientable surfaces. To this end, we will utilize results about the $\infty$-categorical approach to generalized Thom spectra of \cite{ABGHR} and \cite{ABG}. First, we recall the definition of the generalized Thom spectrum of a map $X \to BGL_1(R)$, where $R$ is a ring spectrum. We use $GL_1(R)$ to denote the (grouplike) monoid of units of the ring spectrum $R$. It can be obtained, for example, as the union of $\pi_0$-invertible path components of $\Omega^\infty R$, equivalently the pullback in the following diagram:

$$\xymatrix{
GL_1(R) \ar[r] \ar[d] & \Omega^{\infty} R \ar[d]\\
\pi_0(R)^{\times} \ar[r] & \pi_0(R)
}$$

The papers \cite{ABGHR}, \cite{ABG}, \cite{AF} use quasicategories to model $\infty$-categories. A quasicategory is a weak Kan complex, that is, a simplicial set with inner horn fillers. A functor of quasicategories is simply a map of simplicial sets. An $\infty$-groupoid is a Kan complex. See Chapter 1 of \cite{HTT} for an introduction to quasicategories.

\medskip

In this section and the next, we use this $\infty$-categorical framework, so will not usually specify which $E_n$ or $E_n^\theta$ operad we are working with, unless convenient for a particular application. We use $E_n$ to refer to any operad equivalent to $\mathcal{D}_n$, and $E_n^\theta$ for any operad equivalent to $\oDnt$.

\begin{defn}\label{def-gen-thom}
Let $X$ be a Kan complex. Let $Mod_R$ be the $\infty$-category of (right) $R$-modules, and $Line_R$ its subcategory of free rank 1 cofibrant and fibrant modules, with morphisms equivalences of $R$-modules. The $\infty$-category $Line_R$ is an $\infty$-groupoid, therefore a space, and as such it was shown in \cite{ABG} to be equivalent to $BGL_1(R)$. The Thom spectrum $Th(f)$ of a map $X \to Line_R$ is defined to be the $\infty$-categorical colimit $colim(X \to Line_R \to Mod_R)$.
\end{defn}

\begin{thm}\label{thm-fact-hom-gen}
Let $R$ be a commutative ring spectrum. Let $f: A \to BGL_1(R)$ be an $E_n^{\theta}$-map, and $M^n$ a $\theta$-framed manifold. Then $Th(f)$ is an $E_n^{\theta}$-algebra and $\int_M Th(f)$ is equivalent to the generalized Thom spectrum of the map
$$\int_M f: \int_M A \to \int_M BGL_1(R) \to BGL_1(R)$$
\noindent where $\int_M BGL_1(R) \to BGL_1(R)$ is given by $\int_{M \to pt} (-)$, using the fact that $BGL_1(R)$ is $E_\infty$.

Furthermore, if $M^n$ is a tangentially framed manifold along with a framed embedding $i: M\times \mathbb{R}^{N-n} \hookrightarrow \mathbb{R}^N$, and $A$ is grouplike, then

$$\int_M Th(f) \simeq \Map_c(M,B^n A)^{\phi(f)}$$

\noindent where $\phi(f)$ is obtained as follows:
$$\xymatrix{
\Map_c(M, B^n A) \ar[d]^-{B^nf} \\
\Map_c(M,B^{n+1}GL_1(R)) \ar[d]^-{\wr} \\
\Map_c(M \times \mathbb{R}^{N-n},B^{N+1}GL_1(R))\ar[d]^{i_*} \\
\Map_c(\mathbb{R}^N, B^{N+1}GL_1(R)) \simeq BGL_1(R)
}$$
\end{thm}

\begin{proof}
The generalized Thom spectrum functor is a colimit in the category of $R$-modules; the forgetful functor from $R$-modules to spectra preserves colimits, hence the Thom spectrum functor commutes with colimits in the category of spectra. It also has good multiplicative properties generalizing these Lewis proved for the ordinary Thom spectrum functor: for $R$ a commutative ring spectrum, the generalized Thom spectrum functor is symmetric monoidal (in the $\infty$-categorical sense), see Corollary 8.1 of \cite{ABG}. By Lemma 3.25 of \cite{AF}, these properties guarantee that the generalized Thom spectrum functor commutes with factorization homology. The mapping space description follows from this, along with the fact that the non-abelian Poincar\'e duality equivalence, as in Section 4 of \cite{AF}, is natural in both the manifold and algebra variables.
\end{proof}

As in Section 3, we can describe $\phi(f)$ using the Pontryagin-Thom collapse map associated to the embedding $i$. Again as in Section 3, one can generalize this to stably framed manifolds, and make the same calculations.

\subsection{Calculations}
We calculate factorization homology of Eilenberg-MacLane spectra over oriented surfaces using the following Thom isomorphism theorem, which is Corollary 2.26 in \cite{ABGHR}:

\begin{prop}{(Thom isomorphism theorem, \cite{ABGHR})}\label{prop-thom-iso}
Let $f:X \to BGL_1(R)$ be a map, and suppose $Th(f)$ admits an orientation, that is, a map of right $R$-modules $u: Th(f) \to R$ such that for each $x \in X$, the restriction $Th(f|_x) \to Th(f) \to R$ is a weak equivalence. Then
$$\xymatrix{
Th(f) \ar[r]^-{diag} & Th(f) \wedge X_+ \ar[r]^-{u \wedge id} & R \wedge X_+
}$$
\noindent is a weak equivalence.
\end{prop}

\subsubsection{Factorization homology of $H\mathbb{Z}/p$.}

Exactly as in Lemma \ref{lem-equiv-E_2-fr}, the generalized Thom spectrum $Th(\gamma_p: \Omega^2 S^3 \to BGL_1(S_{(p)}))$ is equivalent to $H\mathbb{Z}/p$ as an $E_2^{SO(2)}$-algebra. The $E_2^{SO(2)}$ structure on the Thom spectrum arises from the fact that $\gamma_p$ is an $E_2^{SO(2)}$-map (with $E_2^{SO(2)}$-algebra structure on $\Omega^2 S^3$ as in the previous section); the structure on $H\mathbb{Z}/p$ comes from its commutative ring spectrum structure.

\medskip 

The Thom spectrum $H\mathbb{Z}/p$ and its factorization homology do not come equipped with a nontrivial map to $S_{(p)}$, so we will use the Thom isomorphism theorem on the Thom spectrum of the composite

$$\Omega^2 S^3 \to BGL_1(S_{(p)}) \to BGL_1(H\mathbb{Z}/p)$$

To describe the Thom spectrum of the composite, we use the following (see, e.g., the introduction to \cite{ABGHR}):

\begin{lem}\label{lem-tensor-thom}
Let $T$ be an $R$-algebra. Then
$$Th(X \to BGL_1(R) \to BGL_1(T)) \simeq Th(X \to BGL_1(R)) \wedge_R T$$
\end{lem}

\begin{cor}\label{cor-thom-compose}
$$Th(\Omega^2 S^3 \to BGL_1(S_{(p)}) \to BGL_1(H\mathbb{Z}/p)) \simeq H\mathbb{Z}/p \wedge_{S_{(p)}} H\mathbb{Z}/p$$
$$Th(\int_M \Omega^2 S^3 \to BGL_1(S_{(p)}) \to BGL_1(H\mathbb{Z}/p)) \simeq (\int_M H\mathbb{Z}/p) \wedge_{S_{(p)}} H\mathbb{Z}/p$$
\end{cor}

Note that $(-)\wedge_{S_{(p)}} H\mathbb{Z}/p = (-) \wedge H\mathbb{Z}/p$.

\medskip

In order to use a Thom isomorphism argument, we need orientations for the Thom spectra in the corollary.

For $Th(\Omega^2 S^3 \to BGL_1(H\mathbb{Z}/p))$, take the multiplication map $mult: H\mathbb{Z}/p \wedge H\mathbb{Z}/p \to H\mathbb{Z}/p$; this is clearly an equivalence when restricted to each $Th(x_0 \hookrightarrow \Omega^2 S^3 \to BGL_1(H\mathbb{Z}/p))$. 

For $Th(\int_M \Omega^2 S^3 \to BGL_1(H\mathbb{Z}/p))$, take the equivalence $\int_M H\mathbb{Z}/p \simeq \int_{M \times \mathbb{R}} H\mathbb{Z}/p$ and compose with the map induced by an embedding $M \times \mathbb{R} \hookrightarrow \mathbb{R}^3$ (for $M$ a closed or punctured genus $g$ surface) to get a map
$$v: \int_M H\mathbb{Z}/p \to \int_{\mathbb{R}^3} H\mathbb{Z}/p \simeq H\mathbb{Z}/p$$
Take as orientation the map
$$\xymatrix{
(\int_M H\mathbb{Z}/p) \wedge H\mathbb{Z}/p \ar[r]^-{v \wedge id} & H\mathbb{Z}/p \wedge H\mathbb{Z}/p \ar[r]^-{mult} & H\mathbb{Z}/p
}$$
This is also an equivalence when restricted to any $Th(x_0 \hookrightarrow \int_M \Omega^2 S^3 \to BGL_1(H\mathbb{Z}/p))$, because composing with an inclusion of any small disc $\mathbb{R}^2 \hookrightarrow M$ gives an equivalence $\int_{\mathbb{R}^2} H\mathbb{Z}/p \simeq \int_{\mathbb{R}^3} H\mathbb{Z}/p$.

\begin{prop}\label{prop-HZ/p}
For $M$ a genus $g$ surface or a punctured genus $g$ surface, Theorem \ref{thm-fact-hom-gen} gives an equivalence
$$\int_M H\mathbb{Z}/p \simeq H\mathbb{Z}/p \wedge \Map_*(M^+ - D^2, S^3)_+$$
\end{prop}

\begin{proof}
Both sides are $H\mathbb{Z}/p$-modules, so as in the $p=2$ case, it suffices to show equivalence after smashing with $H\mathbb{Z}/p$. We obtain this equivalence below.
\begin{align}
H\mathbb{Z}/p \wedge \int_M H\mathbb{Z}/p &\simeq Th(\int_M \Omega^2 S^3 \to BGL_1(H\mathbb{Z}/p)) \\
&\simeq H\mathbb{Z}/p \wedge (\int_M \Omega^2 S^3)_+ \\
&\simeq H\mathbb{Z}/p \wedge \Map_*(M^+,S^3)_+ \\
&\simeq H\mathbb{Z}/p \wedge  (\Omega^2 S^3)_+ \wedge \Map_*(M^+-D^2, S^3)_+ \\
&\simeq H\mathbb{Z}/p \wedge  H\mathbb{Z}/p \wedge \Map_*(M^+-D^2, S^3)_+
\end{align}

(1) holds by Corollary \ref{cor-thom-compose}.

(2) holds by the Thom isomorphism theorem (Proposition \ref{prop-thom-iso}), using the fact that $Th(\int_M \Omega^2 S^3 \to BGL_1(H\mathbb{Z}/p))$ admits an orientation.

(3) holds by Lemma \ref{lem-fact-hom-mahowald}.

(4) holds due to the splitting of the mapping space, as in the proof of Proposition \ref{prop-HZ/2}.

Finally, (5) holds by the Thom isomorphism for $Th(\Omega^2S^3 \to BGL_1(S_{(p)}) \to BGL_1(H\mathbb{Z}/p)$), and we can conclude.
\end{proof}

\subsubsection{Factorization homology of $H\mathbb{Z}_{(p)}$.}

We will calculate $\int_M H\mathbb{Z}_{(p)}$ for $M$ a closed or punctured genus $g$ surface, where $\mathbb{Z}_{(p)}$ denotes the $p$-local integers, using the fact that $H\mathbb{Z}_{(p)}$ is the Thom spectrum of a map 
$$\beta: \Omega^2(S^3\langle3\rangle) \to BSL_1(S_{(p)})$$
\noindent where $S^3\langle3\rangle$ denotes the 3-connected cover of $S^3$; the map $S^3\langle3\rangle \to S^3$ induces isomorphism on all $\pi_i$, $i>3$, and $\pi_i(S^3\langle3\rangle)=0$ for $i \leq 3$. The space $SL_1(S_{(p)})$ is the identity component of $GL_1(S_{(p)})$.

As in \cite{CMT} or Section 9 of \cite{Blu}, this equivalence is given by the composite

$$\xymatrix{
Th(\beta) \ar[r]^-{\beta_*}  & MSL_1(S_{(p)}) \ar[r]^-{Thom} & H\mathbb{Z}_{(p)}
}$$

The Thom spectrum $MSL_1(S_{(p)})$ has a Thom class to $H\mathbb{Z}_{(p)}$ because $BSL_1(S_{(p)})$ is simply-connected; it is the 1-connected cover of $BGL_1(S_{(p)})$. The map 
$$\beta: \Omega^2(S^3\langle3\rangle) \to BSL_1(S_{(p)})$$
\noindent is given by lifting the composite

$$\xymatrix{
\Omega^2(S^3\langle3\rangle) \ar[r] & \Omega^2 S^3 \ar[r]^-{\gamma_p} & BGL_1(S_{(p)})
}$$
\noindent to $BSL_1(S_{(p)}) = BGL_1(S_{(p)})\langle1\rangle$.

\begin{lem}\label{lem-gp-action-cover}
The map $\beta$ is an $E_2 ^{SO(2)}$-map, where the $E_2^{SO(2)}$-action on $\Omega^2 S^3$ is as in Section 3. With respect to this structure and the trivial $E_2^{SO(2)}$-action on $H\mathbb{Z}_{(p)}$, 
$$Th(\beta) \simeq H\mathbb{Z}_{(p)}$$
\noindent as $E_2^{SO(2)}$-algebras.
\end{lem}

\begin{proof}
The Thom class is a map of $E_2^{SO(2)}$-algebras with trivial $SO(2)$-action on both spectra. Thus, it is enough to realize $\beta$ as an $E_2^{SO(2)}$-map, where $BSL_1(S_{(p)})$ is given the trivial action. Consider the following commutative diagram:

$$\xymatrix{
(D_2 \times L)S^1 \ar[r]^-{\gamma_p} \ar[d]^{act} & BGL_1(S_{(p)}) \ar[d] \\
S^1 \ar[r] & B\mathbb{Z}^{\times}_{(p)}
}$$

All of the maps in this diagram are $\mathcal{D}_2 \times \mathcal{L}$-maps, as well as $SO(2)$-maps (with respect to the Salvatore-Wahl action on $(D_2 \times L)S^1$, and the trivial action on the other spaces). Thus, the induced map between the homotopy fibers of the vertical maps is also such. This realizes $\beta$ as an $E_2^{SO(2)}$-map.
\end{proof}

By Theorem \ref{thm-fact-hom-gen}, $\int_M H\mathbb{Z}_{(p)}$ is a generalized Thom spectrum of a virtual bundle over $\int_M \Omega^2 (S^3\langle3\rangle)$, where the $E_2^{SO(2)}$-action on $\Omega^2 (S^3\langle3\rangle)$ is such that the map to $\Omega^2 S^3$ is an $E_2^{SO(2)}$-algebra map. That is, the map $S^3\langle3\rangle \to S^3$ is $SO(2)$-equivariant. The following gives an easy description of $\int_M H\mathbb{Z}_{(p)}$.

\begin{prop}\label{prop-HZ(p)}
Let $M$ be a closed genus $g$ surface or a punctured genus $g$ surface. Then
$$\int_M H\mathbb{Z}_{(p)} \simeq H\mathbb{Z}_{(p)} \wedge \Map_*(M^+ - D^2, S^3,\langle3\rangle)_+$$
In particular,
$$\int_{\Sigma_g} H\mathbb{Z}_{(p)} \simeq H\mathbb{Z}_{(p)} \wedge S^3\langle3\rangle_+ \wedge (\Omega(S^3\langle3\rangle))^{2g}_+$$
\end{prop}

\begin{proof}
We first prove that $\int_M \Omega^2 (S^3\langle3\rangle) \simeq \Map_c(M, S^3\langle3\rangle)$. Punctured genus $g$ surfaces have trivial tangent bundles, so it remains to show this for $M$ a closed genus $g$ surface. For such $M$, $\int_M \Omega^2 (S^3\langle3\rangle) \simeq \Gamma(B^{TM} \Omega^2(S^3\langle3\rangle))$. Recall that, from Lemma \ref{lem-fact-hom-mahowald}, $\Gamma(B^{TM} \Omega^2 S^3) \simeq \Map(M,S^3)$, and that
$$\Map(M, S^3) \simeq S^3 \times (\Omega S^3)^{2g} \times \Omega^2 S^3$$

Hence the map $\Map(M, S^3\langle3\rangle) \to \Map(M, S^3)$ has the effect of killing $\pi_1(\Map(M, S^3))$ and the $\mathbb{Z}$-summands in $\pi_2$ and $\pi_3$ ($2g$ summands in $\pi_2$, one in $\pi_3$). It is an isomorphism on the rest of $\pi_*$. Thus, in order to show  
$$\Gamma(B^{TM} \Omega^2(S^3\langle3\rangle)) \simeq \Map(M, S^3\langle3\rangle)$$
\noindent it suffices to show that the map 
$$\Gamma(B^{TM} \Omega^2(S^3\langle3\rangle)) \to \Gamma(B^{TM} \Omega^2 S^3)$$
\noindent has this effect on homotopy groups as well. In order to prove this, consider the commutative diagram of fiber sequences

$$\xymatrix{
\Gamma(B^{TM} \Omega^2(S^3\langle3\rangle)) \ar[r] \ar[d] & \Gamma(B^{TM} \Omega^2 S^3) \ar[d] \\
\Map(M, B^{TM} \Omega^2(S^3\langle3\rangle)) \ar[r] \ar[d] & \Map(M,B^{TM} \Omega^2 S^3) \ar[d] \\
\Map(M,M) \ar[r] & \Map(M,M)
}$$
\noindent where the fiber is taken over the identity map of $M$. As the map on base spaces is the identity, it suffices to show that the map 
$$\Map(M, B^{TM} \Omega^2(S^3\langle3\rangle)) \to \Map(M,B^{TM} \Omega^2 S^3)$$
\noindent annihilates $\pi_1(\Map(M,B^{TM} \Omega^2 S^3))$ and the $\mathbb{Z}$-summands in $\pi_2$ and $\pi_3$, and induces isomorphism on the rest of $\pi_*$. This follows by considering the following commutative diagram of fiber sequences

$$\xymatrix{
\Map(M, S^3\langle3\rangle) \ar[r] \ar[d] & \Map(M,S^3) \ar[d] \\
\Map(M, B^{TM} \Omega^2(S^3\langle3\rangle)) \ar[r] \ar[d] &\Map(M, B^{TM} \Omega^2 S^3)  \ar[d]\\
\Map(M,M) \ar[r] & \Map(M,M)
}$$
This time, the fiber is taken over a constant map $M \to M$. Again the map of base spaces is the identity. The map on fibers annihilates $\pi_1(\Map(M, S^3))$ and the $\mathbb{Z}$-summands in $\pi_2$ and $\pi_3$, and induces isomorphism on the rest of $\pi_*$, thus the map on total spaces does the same. This concludes the proof that

$$\int_M \Omega^2 (S^3\langle3\rangle) \simeq \Map_c(M, S^3\langle3\rangle)$$

\medskip

To prove that $\int_M H\mathbb{Z}_{(p)} \simeq H\mathbb{Z}_{(p)} \wedge \Map_*(M^+-D^2, S^3\langle3\rangle)$, consider the composite below. The splitting of the mapping space occurs because $S^3\langle3\rangle$ is an H-space. The Thom isomorphism holds because the Thom spectrum $\int_M H\mathbb{Z}_{(p)}$ admits an $H\mathbb{Z}_{(p)}$-orientation, exactly as in the $H\mathbb{Z}/p$ case (or because $\Map_*(M^+, S^3\langle 3 \rangle)$ is simply connected).

$$\xymatrix{
\int_M H\mathbb{Z}_{(p)} \ar[d]^-{1 \wedge id} \\
H\mathbb{Z}_{(p)} \wedge \int_M H\mathbb{Z}_{(p)} \ar[d]^-{Thom}_-{\wr} \\
H\mathbb{Z}_{(p)} \wedge \Map_*(M^+, S^3\langle 3\rangle)_+ \ar[d]_-{\wr} \\
H\mathbb{Z}_{(p)} \wedge \Omega^2(S^3\langle3\rangle)_+ \wedge \Map_*(M^+ -D^2, S^3\langle3\rangle)_+ \ar[d]^-{Thom}_-{\wr} \\
H\mathbb{Z}_{(p)} \wedge H\mathbb{Z}_{(p)} \wedge \Map_*(M^+ -D^2, S^3\langle 3 \rangle)_+ \ar[d]^-{mult \wedge id} \\
H\mathbb{Z}_{(p)} \wedge \Map_*(M^+ -D^2, S^3\langle 3 \rangle)_+
}$$

To see that this is an equivalence, it suffices to show that it induces isomorphism on homology with coefficients in $\mathbb{Z}_{(p)}$, because both spectra in question are $H\mathbb{Z}_{(p)}$-modules (as in the $H\mathbb{Z}/2$ case, $\int_M H\mathbb{Z}_{(p)}$ is even an $H\mathbb{Z}_{(p)}$-algebra). By the Thom isomorphism, the spectra $\int_M H\mathbb{Z}_{(p)}$ and $H\mathbb{Z}_{(p)} \wedge \Map_*(M^+ -D^2, S^3\langle 3 \rangle)_+$ both have $\mathbb{Z}_{(p)}$-homology

$$H_*(\Map_*(M^+, S^3\langle3\rangle); \mathbb{Z}_{(p)}) \cong H_*(H\mathbb{Z}_{(p)}; \mathbb{Z}_{(p)}) \otimes H_*(\Map_*(M^+ -D^2, S^3\langle 3 \rangle); \mathbb{Z}_{(p)})$$

The middle four spectra in the sequence of maps above have $\mathbb{Z}_{(p)}$-homology

$$H_*(H\mathbb{Z}_{(p)}; \mathbb{Z}_{(p)}) \otimes H_*(H\mathbb{Z}_{(p)}; \mathbb{Z}_{(p)}) \otimes H_*(\Map_*(M^+ -D^2, S^3\langle 3 \rangle); \mathbb{Z}_{(p)})$$

The first map in the sequence sends

$$\Sigma_i a_i \otimes b_i \in H_*(H\mathbb{Z}_{(p)}; \mathbb{Z}_{(p)}) \otimes H_*(\Map_*(M^+ -D^2, S^3\langle 3 \rangle); \mathbb{Z}_{(p)})$$

\noindent to $1 \otimes (\Sigma_i a_i \otimes b_i)$. The next three maps are isomorphisms that preserve the tensor factors. The last map, $mult \wedge id$, then takes $1 \otimes (\Sigma_i a_i \otimes b_i)$ to $\Sigma_i a_i \otimes b_i$. Thus the composite induces isomorphism on $\mathbb{Z}_{(p)}$-homology. This concludes the proof.

\end{proof}

This allows us to calculate the factorization homology of $H\mathbb{Z}$ over closed and punctured genus $g$ surfaces.

\begin{cor}\label{cor-HZ}
For $M$ a closed or punctured genus $g$ surface,
$$\int_M H\mathbb{Z} \simeq H\mathbb{Z} \wedge \Map_*(M^+ - D^2, S^3\langle3\rangle)_+$$
\end{cor}

\begin{proof}
We first show that the localizations at each prime $p$ of these two spectra are equivalent. Localization at $p$ is given by smashing with the $p$-local sphere spectrum, thus it commutes with colimits and is symmetric monoidal (in the $\infty$-categorical sense), see Lemma 3.4 of \cite{GGN} or Proposition 2.2.1.9 of \cite{HA}. Hence by Lemma 3.25 of \cite{AF}, localization at $p$ commutes with factorization homology and

$$(\int_M H\mathbb{Z})_{(p)} \simeq \int_M H\mathbb{Z}_{(p)} \simeq H\mathbb{Z}_{(p)} \wedge \Map_*(M^+ - D^2, S^3\langle3\rangle)_+$$

The spectrum $H\mathbb{Z} \wedge \Map_*(M^+ - D^2, S^3\langle3\rangle)_+$ is rationally trivial because $\Map_*(M^+ - D^2, S^3\langle3\rangle)$ is a product of copies of $\Omega (S^3\langle3\rangle)$, or a product of $S^3\langle3\rangle$ with copies of $\Omega (S^3\langle3\rangle)$, which are rationally trivial. The spectrum $\int_M H\mathbb{Z}$ is also rationally trivial: the rationalization of $\int_M H\mathbb{Z}$ is equivalent to the rationalization of $(\int_M H\mathbb{Z})_{(p)} \simeq \int_M H\mathbb{Z}_{(p)}$, which is trivial, again because $\Map_*(M^+ - D^2, S^3\langle3\rangle)$ is rationally trivial. Every rationally trivial spectrum decomposes into its $p$-primary parts, and we have shown that the $p$-localizations of $\int_M H\mathbb{Z}$ and $H\mathbb{Z} \wedge \Map_*(M^+ - D^2, S^3\langle3\rangle)_+$ agree; thus they are equivalent.
\end{proof}

\begin{cor}\label{cor-higher-THH-HZ}
$$THH_*^{S^2}(H\mathbb{Z}) \cong H_*(S^3\langle3\rangle)$$
$$THH_*^{\mathbb{T}^2}(H\mathbb{Z}) \cong H_*(S^3\langle3\rangle) \otimes H_*(\Omega(S^3\langle3\rangle))^{\otimes 2}$$

$H_*(S^3\langle3\rangle)$ is 0 in odd dimensions, and $\mathbb{Z}/n\mathbb{Z}$ in dimension $2n$.

$H_*(\Omega(S^3\langle3\rangle))$ is 0 in positive even dimensions, and $\mathbb{Z}/n\mathbb{Z}$ in dimension $2n-1$.
\end{cor}

\section{Hochschild cohomology of Thom spectra}

In this section, we use Theorem \ref{thm-main} and Theorem \ref{thm-fact-hom-gen}, and the relation between factorization homology and $E_n$ Hochschild cohomology, to describe the $E_n$ Hochschild cohomology of Thom spectra. As in the previous section, we are working in the $\infty$-categorical context, and use $E_n$ to refer to any operad equivalent to $\mathcal{D}_n$.

\begin{defn}\label{def-U(A)}
Let $A$ be an $E_n$-algebra. Denote by $U(A)$ the $E_1$-algebra $\int_{S^{n-1} \times \mathbb{R}} A$; the $E_1$-algebra structure is induced by the action of embeddings on the $\mathbb{R}$-coordinate, $S^{n-1} \times (\coprod_k \mathbb{R}) \hookrightarrow S^{n-1} \times \mathbb{R}$. In some of the literature, $U(A)$ is referred to as the enveloping algebra of $A$ (see, e.g., Proposition 3.16 of \cite{Fra}).
\end{defn}

As in Section 3 of \cite{Fra}, left $U(A)$-modules are $E_n$-$A$-modules. For example, $A = \int_{\mathbb{R}^n} A$ is a left module over $U(A)$. This action can be roughly described as induced by an embedding $(S^{n-1} \times \mathbb{R}) \coprod \mathbb{R}^n \hookrightarrow \mathbb{R}^n$. In this embedding, $S^{n-1} \times \mathbb{R}$ is an annulus around the open disk $\mathbb{R}^n$, as in Figure 1 below. In this figure, the annulus and the disk are labeled by $A$, in keeping with the intuition of factorization homology as a labeled configuration space. For more details about this action, see Sections 2 and 3 of \cite{Fra}, Section 5 of \cite{GTZ12}, or Section 2 of \cite{Hor16}.

\begin{figure}[h]
\includegraphics[scale=0.5]{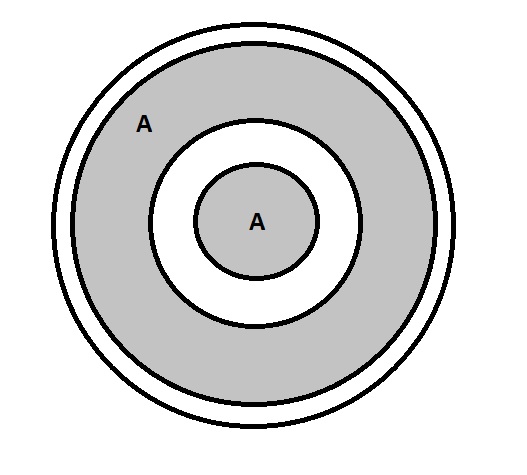}
\centering
\caption{The action of $U(A)$ on $A$}
\end{figure}

\begin{defn}\label{def-higher-THH*}
For an $E_n$-algebra $A$ in spectra, define its $E_n$ Hochschild cohomology $THC_{E_n}(A,A) = Rhom_{U(A)}(A,A)$. 
\end{defn}

Equivalently, this is $Rhom_{E_n\text{-}A\text{-}mod}(A,A)$, the mapping spectrum in the category of $E_n$-$A$ modules (see, e.g., Proposition 3.19 of \cite{Hor16}). 

\begin{rem}\label{rem-THH*}
For $n=1$, this recovers Hochschild cohomology. For $n=1$, $S^{n-1} \times \mathbb{R} = S^0 \times \mathbb{R}$ and
$$\int_{S^0 \times \mathbb{R}} A \simeq A \wedge A^{op}$$
\noindent as an algebra. Under this identification, the action of $\int_{S^0 \times \mathbb{R}} A$ on $A$ as its enveloping algebra agrees with the action of $A \wedge A^{op}$ on $A$. Thus, for $n=1$, the definition of higher Hochschild cohomology above reduces to $Rhom_{A \wedge A^{op}}(A,A)$, which is Hochschild cohomology.
\end{rem}

Note the the following terminology issue: higher Hochschild cohomology is not what Ayala-Francis call factorization cohomology in \cite{AF2}. Factorization cohomology is an invariant of $E_n$-coalgebras, and in that sense is a generalization of coHochschild homology of coalgebras (see, e.g., \cite{HPS}), rather than of Hochschild cohomology of algebras. Factorization cohomology is can be thought of as cohomology for manifolds, whereas Hochschild cohomology can be thought of as cohomology for algebras.

\begin{defn}\label{def-higher-HH}
Let $A$ be an $E_n$-algebra in spectra. The $E_n$ Hochschild homology of $A$, $THH^{E_n}(A)$, is defined as the derived smash product $A \wedge^L _{U(A)} A$, where $U(A)$ acts on $A$ on the right using an equivalence $U(A) \simeq U(A)^{op}$, see Lemma 3.20 of \cite{Fra}.
\end{defn}

For $n=1$, this recovers Hochschild homology. If $A$ is $E_{n+1}$, $THH^{E_n}(A) \simeq \int_{S^n \times \mathbb{R}} A$.

\begin{rem}\label{rem-higher-HH-FH}
In general, higher topological Hochschild homology need not coincide with factorization homology. For $n \neq 1,3,7$, $S^n$ is not framed and thus $\int_{S^n} A$ is only defined for $A$ an $E_n ^{SO(n)}$ algebra. Higher Hochschild homology is defined for any $E_n$-algebra, and in particular does not depend on the data of an $SO(n)$-action.

An example in which the two differ is given, for $n=2$, by $A = \Sigma^\infty _+ \Omega^2 \Sigma^2 (S^0 \vee S^0)$. This can be given an $E_2^{SO(2)}$-algebra structure by specifying that $SO(2)$ acts on the $\Sigma^2$ coordinate as it acts on $(\mathbb{R}^2)^+$. Then by Theorem 9.4 of \cite{Hor},

$$THH^{E_2}(A) \simeq \Sigma^\infty _+ \Map(S^2, S^2 \vee S^2)$$

On the other hand,

$$\int_{S^2} A \simeq \Sigma^\infty _+ \Gamma(\tau_2^+(S^0 \vee S^0))$$

\noindent where $\tau_2^+$ denotes the fiberwise one-point compactification of $TS^2$, and $\tau_2^+(S^0 \vee S^0)$ is the fiberwise smash product of this bundle with $S^0 \vee S^0$. In Example 15 of \cite{Sal04}, Salvatore shows that $\Map(S^2, S^2 \vee S^2)$ and $\Gamma(\tau_2^+(S^0 \vee S^0))$ do not have the same rational homology.

One can think of factorization homology as keeping track of the tangent bundle of $S^2$, whereas higher Hochschild homology only sees a trivial bundle on $S^2$.
\end{rem}

Now let $X$ be a pointed, $(n-1)$-connected space, and let $f: X \to B^{n+1}G$ be a pointed map, where $G$ is one of the stabilized Lie groups ($O$,$U$,$SO$,...) or $G = GL_1(R)$ for $R$ a commutative ring spectrum. Let $A = \Omega^n X ^{\Omega^n f}$ be the Thom spectrum of $\Omega^n f$, hence an $E_n$-ring spectrum.

Recall that $\Omega X$ acts on $\Omega^n X$, giving an action of $\pi_1(X)$ on $\pi_n(X)$. For $n=1$, $\pi_1(X)$ acts on itself by conjugation, as in Figure 2:

\begin{figure}[h]
\includegraphics[scale=0.5]{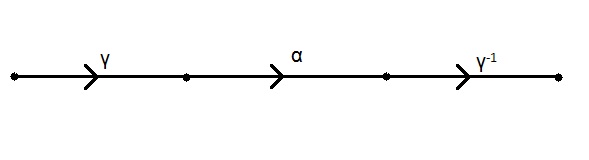}
\centering
\caption{The conjugation action of $\Omega X$ on itself}
\end{figure}

For $n>1$, $\pi_1(X)$ acts on $\pi_n(X)$ via a generalized conjugation action, as in Figure 3 below.

\begin{figure}[h]
\includegraphics[scale=0.5]{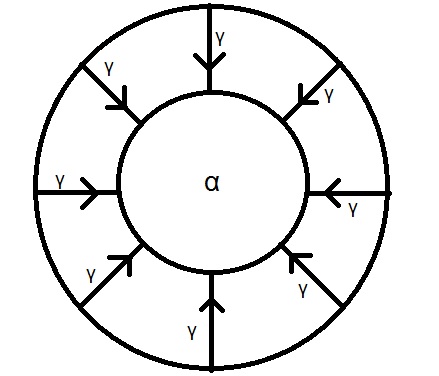}
\centering
\caption{The generalized conjugation action of $\Omega X$ on $\Omega^n X$}
\end{figure}

This action can also be thought of as follows: include $\Omega X$ into $\Omega \Map(S^{n-1},X) = \Map_c(S^{n-1} \times \mathbb{R}, X)$ via the inclusion $X \to \Map(S^{n-1}, X)$ of constant maps. This is a loop map. The mapping space $\Map_c(S^{n-1} \times \mathbb{R}, X) = \int_{S^{n-1} \times \mathbb{R}} \Omega^n X$ acts on $\Omega^n X$ as in Figure 1.

We generalize this action to Thom spectra. That is, we construct an action of $\Sigma^\infty _+ \Omega X$ on $\Omega^n X ^{\Omega^n f}$, which specializes to the generalized conjugation action above if $f$ is a trivial map. This action arises from a map of ring spectra $\Sigma^\infty _+ \Omega X \to \int_{S^{n-1} \times \mathbb{R}} \Omega^n X ^{\Omega^n f}$ as follows:

\medskip

\textbf{Construction.} Let $c: \Omega X \to \Omega \Map(S^{n-1}, X)$ be the homotopy fiber of the map $i: \Omega \Map(S^{n-1},X) \to \Omega^n X$, obtained by considering $\Omega \Map(S^{n-1}, X)$ as maps $(D^n, \partial D^n) \to (X, *)$ that send $0 \in D^n$ to the basepoint. Note that $i$ is not a fibration; even when $\Omega^n X$ is path connected, $i$ is not surjective, as every $\alpha: D^n \to X$ in its image evaluates to the basepoint on $0 \in D^n$. 

\begin{lem}\label{lem-ring-map} On Thom spectra, $c$ induces a map of ring spectra

$$\mathbf{c}: \Sigma^\infty _+ \Omega X \to \int_{S^{n-1} \times \mathbb{R}} \Omega^n X ^{\Omega^n f} = U(\Omega^n X ^{\Omega^n f})$$

\end{lem}

\begin{proof} Note that $c$ is an $E_1$-map; it can obtained by looping the map $X \to \Map(S^{n-1}, X)$ obtained by inclusion of constant maps. This can be seen by observing that $i: \Omega \Map(S^{n-1}, X) \to \Omega^n X$ is the fiber of the map $ev_0 : \Omega^n X \to X$, which evaluates a map from $D^n$ at 0, and that $ev_0$ is the homotopy fiber of the map $X \to \Map(S^{n-1}, X)$ above. (When $n=1$, $c$ is the map $\Omega X \to \Omega X \times \Omega X$ which sends $g$ to $(g, g^{-1})$.)

Because $c$ is an $E_1$-map, it induces, for each $E_1$-map $h: \Omega \Map(S^{n-1}, X) \to BG$, a map of ring spectra $Th(h \circ c) \to Th(h)$. We take $h = \Omega^n f \circ i$. Note that by Theorem \ref{thm-main} (or Theorem \ref{thm-fact-hom-gen}), $Th(h) \simeq \int_{S^{n-1} \times \mathbb{R}} \Omega^n X ^{\Omega^n f}$, because $S^{n-1} \times \mathbb{R} \hookrightarrow \mathbb{R}^n$ is a framed embedding and the diagram

$$\xymatrix{
\Omega \Map(S^{n-1}, X) \ar[r]^-i \ar[d]^-{\Omega^n f} & \Omega^n X \ar[d]^-{\Omega^n f} \\
\Omega \Map(S^{n-1}, B^{n+1} G) \ar[r]^-i & \Omega^n B^{n+1} G \\
}$$

\noindent commutes. Thus we obtain a map of ring spectra $Th(\Omega^n f \circ i \circ c) \to \int_{S^{n-1} \times \mathbb{R}} \Omega^n X ^{\Omega^n f}$. By construction, $i \circ c$ is nullhomotopic, and furthermore $\Omega^n f \circ i \circ c$ is nullhomotopic via $E_1$-maps; this composite can be written as

$$\xymatrix{
\Omega X \ar[r]^-{\Omega f} & \Omega B^{n+1} G \ar[r] & \Omega^n B^{n+1} G
}$$

\noindent where the last map is induced by a nullhomotopic map $S^1 \to S^n$. $B^{n+1}G$ is an infinite loop space, thus this map is via $E_1$-maps. Therefore $Th(\Omega^n f \circ i \circ c) \simeq \Sigma^\infty _+ \Omega X$ as ring spectra, and $c$ induces a map of ring spectra $\Sigma^\infty _+ \Omega X \to \int_{S^{n-1} \times \mathbb{R}} \Omega^n X ^{\Omega^n f}$.

\end{proof}

The following proof was inspired by Beardsley's approach to relative Thom spectra in \cite{Bea17}.

\begin{thm}\label{thm-conj-rhom}
Under the generalized conjugation action of $\Sigma^\infty_+ \Omega X$ on $A =\Omega^n X ^{\Omega^n f}$, coming from the ring map $\mathbf{c}: \Sigma^\infty _+ \Omega X \to U(A)$, there are equivalences
$$THH^{E_n}(A,A) \simeq S \wedge^L_{\Sigma^\infty _+ \Omega X} A$$
$$THC_{E_n}(A,A) \simeq Rhom_{\Sigma^\infty _+ \Omega X}(S, A)$$
\end{thm}

\begin{proof}
It suffices to show that

$$A \simeq B(S, \Sigma^\infty _+ \Omega X, U(A))$$

\noindent as $U(A)$-modules; this implies that 

$$THC_{E_n}(A) = Rhom_{U(A)}(A,A) \simeq Rhom_{\Sigma^\infty _+ \Omega X} (S, A)$$

\noindent and similarly for higher topological Hochschild homology.

\medskip

By Lemma \ref{lem-ring-map}, we have a principal fibration

$$\xymatrix{
\Omega X \ar[r]^-c & \Omega \Map(S^{n-1},X) \ar[r]^-i & \Omega^n X
}$$

\noindent in which $c$ is a loop map, and $\Omega^n X \simeq \Omega \Map(S^{n-1}, X)// \Omega X$. We show that this statement about homotopy orbits descends to Thom spectra. We replace $i:\Omega \Map(S^{n-1}, X) \to \Omega^n X$ with a fibration, which we will denote $ \tilde{i} :\tilde {\Omega} \Map(S^{n-1}, X) \to \Omega^n X$. This results in a fiber sequence

$$\tilde{\Omega} X \to \tilde{\Omega} \Map(S^{n-1}, X) \to \Omega^n X$$

\noindent in which the composite $\tilde{\Omega} X \to \Omega^n X$ is trivial, and $\Omega^n X = \tilde {\Omega} \Map(S^{n-1},X) / \tilde{\Omega} X$.

\medskip

We give explicit models for $\tilde {\Omega} \Map(S^{n-1},X)$ and $\tilde{\Omega} X$, in which it is easy to see that $\tilde{\Omega}X$ is still an algebra, and $\tilde{\Omega} \Map(S^{n-1},X)$ is still a module over it. Set $\tilde{\Omega} X$ to be $\Omega \Map(I^n, X)$, thought of as the space of maps $\phi: I^{n+1} \to X$ such that $\phi(0,s) = \phi(1,s) = *$ for all $s \in I^n$. Let $\tilde{\Omega}\Map(S^{n-1},X)$ be the space of maps $\phi: I^{n+1} \to X$ such that $\phi(1,s) = *$ for all $s \in I^n$, and $\phi(0,-) \in \Omega^n X$. Then $\tilde{\Omega} X$ includes into $\tilde{\Omega} \Map(S^{n-1}, X)$, and acts on it by concatenation in the first $I$ direction (on the $1 \in I$ end.) The map $\tilde{i}: \tilde{\Omega} \Map(S^{n-1}, X) \to \Omega^n X$ is given by evaluation on $0 \in I$.

\medskip

Then our new model for $\int_{S^{n-1} \times \mathbb{R}} A$ is $Th(\Omega^n f \circ \tilde{i})$. The composite $\tilde{\Omega} X \to \Omega^n X$ is trivial, thus the Thom spectrum of $\tilde{\Omega} X \to \Omega^n X \to BG$ is $\Sigma^\infty _+ \tilde{\Omega} X$. The action of $\tilde{\Omega} X$ on $\tilde{\Omega} \Map(S^{n-1}, X)$ is compatible with the maps to $BG$, thus descends to an action of $\Sigma^\infty _+ \tilde{\Omega} X$ on $\int_{S^{n-1} \times \mathbb{R}} A$. 

\medskip

By Corollary 8.1 of \cite{ABG}, the Thom spectrum functor is symmetric monoidal and preserves colimits. Thus $B(Th((\Omega^n f) \circ \tilde{i}), \Sigma^\infty _+ \tilde{\Omega} X, S) \simeq B(U(A), \Sigma^\infty_+ \Omega X, S)$ is the Thom spectrum of the map

$$\xymatrix{
B(\tilde {\Omega} \Map(S^{n-1},X), \tilde{\Omega} X, *) \ar[d] \\
B(\tilde {\Omega} \Map(S^{n-1},B^{n+1} G), \tilde{\Omega} B^{n+1}G, *) \ar[d]^-{(\tilde i,*,*)} \\
B(\Omega^n B^{n+1} G, *,*) \ar[d]^-{\wr} \\
\Omega^n B^{n+1} G = B(\Omega^n B^{n+1} G,\Omega^n B^{n+1} G,\Omega^n B^{n+1} G)
}$$

This fits into a commutative diagram

$$\xymatrix{
B(\tilde {\Omega} \Map(S^{n-1},X), \tilde{\Omega} X, *) \ar[r]^-{\sim} \ar[d] & \Omega^n X \ar[d]^-{\Omega^n f} \\
\Omega^n B^{n+1} G \ar[r]^{=} & \Omega^n B^{n+1} G
}$$

Thus

$$B(Th((\Omega^n f) \circ \tilde{i}), \Sigma^\infty _+ \tilde{\Omega} X, S) \simeq A$$

\noindent and we obtain an equivalence $B(U(A), \Sigma^\infty _+ \Omega X, S) \to A$, as required. As the map $U(A) \to A$ induced by $i: \Omega \Map(S^{n-1},X) \to \Omega^n X$ is a map of $U(A)$-modules, so is this equivalence.
\end{proof}

This theorem is analogous to the description of Hochschild cohomology of Hopf $k$-algebras with antipode, which is given by $Rhom_A(k,A)$, with $A$ acting on itself by conjugation, see, e.g., Section 3.2 and Proposition A.1.8 of \cite{Mal}.

\bigskip

We now interpret this theorem through the lens of parametrized spectra and Atiyah duality.

\medskip

Given spaces $X$ and $F$, an $F$-bundle over $X$ is obtained from an action of $\Omega X$ on $F$; the total space is then given by $E\Omega X \times_{\Omega X} F$. Similarly, a parametrized spectrum over $X$, with fiber spectrum $\mathbf{F}$, is obtained from an action of $\Omega X$ on $\mathbf{F}$. For an overview of parametrized spectra and their homology and cohomology, see, e.g., Section 1 of \cite{CJ13}.

The action of $\Sigma_+ ^\infty \Omega X$ on $A = \Omega^n X ^{\Omega^n f}$ thus gives a  a parametrized spectrum over $X$ with fiber spectrum $A$. Its cohomology spectrum, equivalently spectrum of sections, is $Rhom_{\Sigma_+ ^\infty \Omega X}(S,A)$, which by Theorem \ref{thm-conj-rhom} is equivalent to $THC_{E_n}(A,A)$. Its homology spectrum is equivalent to the derived smash product $B(A, \Sigma_+ ^\infty \Omega X, S)$, and this is equivalent to $THH^{E_n}(A)$ by Theorem \ref{thm-conj-rhom}.

\medskip

We restrict to the case in which $X = M$ is a closed, connected manifold. By Atiyah duality for parametrized spectra (e.g., Theorem 10 of \cite{CJ13}), the cohomology spectrum, that is, the section spectrum $THC_{E_n}(A,A)$, can be obtained from the homology spectrum via twisting by $-TM$, the stable normal bundle to $M$. More explicitly, if $THH^{E_n}(A)$ can be obtained as a Thom spectrum $\Map(S^n,M)^{l^n(f)}$ (this is the case when $THH^{E_n}$ agrees with factorization homology), we take the Whitney sum of $l^n(f)$ with the virtual bundle $-TM$. Recall from Proposition \ref{prop-higher-hopf} that for $n=1,3,7$, the description of $l^n(f)$ involves the Hopf map and its higher dimensional analogues.

Even if $THH^{E_n}(A)$ is not necessarily a Thom spectrum of a virtual bundle over the mapping space, it is still a Thom spectrum over $M$; that is, the homology spectrum of a parametrized spectrum over $M$. As such, we can still twist it by $-TM$, which can also be seen as coming from a parametrized spectrum over $M$. In summary,

\begin{cor}\label{cor-atiyah-duality}
For $A= \Omega^n M^{\Omega^n f}$, the $E_n$ Hochschild cohomology $THC_{E_n}(A)$ is a twisting by $-TM$ of the $E_n$ Hochschild homology $THH^{E_n}(A)$. Therefore, if $l^n(f): \Map(S^n, M) \to BG$ is such that $THH^{E_n}(A) \simeq \Map(S^n, M)^{l^n(f)}$, then
$$THC_{E_n}(A) \simeq \Map(S^n,M)^{l^n(f)-TM}$$
\end{cor}

\begin{rem}\label{rem-string-topology}
For $A= \Sigma^\infty _+ \Omega^n M$, this gives 
$$THC_{E_n}(\Sigma^\infty _+ \Omega^n M) \simeq \Map(S^n, M)^{-TM}$$
\noindent and the string topology operations on $H_*(\Map(S^n,M))$, see, e.g., Theorem 1 of \cite{CJ02} for $n=1$ and Theorem 1.2 of \cite{Hu} or Theorem 7.1 of \cite{GTZ12} for higher $n$.
\end{rem}

The spectrum $THC_{E_n}(A)$ is an $E_{n+1}$-algebra by the higher Deligne conjecture (see, e.g., Theorem 6.28 of \cite{GTZ12}).

\begin{cor}\label{cor-En+1}
The spectrum $\Map(S^n,M)^{l^n(f)-TM}$ is an $E_{n+1}$-ring spectrum.
\end{cor}

\subsection{Calculations}

\begin{prop}\label{prop-triv-action}
If $f: X \to B^{n+1} G$ is an $E_1$-map, then the action of $\Sigma^\infty _+ \Omega X$ on $A= \Omega^n X ^{\Omega^n f}$ is homotopically trivial. That is, the map $GL_1(\Sigma^\infty _+ \Omega X) \to hAut(A)$ is nullhomotopic via $E_1$-maps.
\end{prop}

\begin{proof}
Recall, from the proof of Theorem \ref{thm-conj-rhom}, that $\Sigma^\infty _+ \Omega X$ is the Thom spectrum of the map

$$\xymatrix{
\Omega X \ar[r] & \Omega \Map(S^{n-1}, X) \ar[r]^-{i} & \Omega^n X \ar[r]^{\Omega^n f} & BG
}$$

If $X$ is $E_1$, then the composite $\Omega X \to \Omega^n X$ is an $E_1$-map, which is furthermore nullhomotopic via $E_1$-maps, as its nullhomotopy is induced by a nullhomotopy of a map $S^n \to S^1$. If $f$ is an $E_1$-map, then the resulting map on Thom spectra, $\Sigma^\infty _+ \Omega X \to A$, is nullhomotopic via $E_1$-maps. Note that the action $\Sigma^\infty_+ \Omega X \wedge A \to A$ does not, in general, factor as $\Sigma^\infty_+ \Omega X \wedge A \to A \wedge A \to A$. For example, if $A = \Sigma^\infty _+ \Omega X$, the conjugation action does not factor in this way. In this case, however, $A$ is $E_{n+1}$, so the action of $\int_{S^{n-1} \times \mathbb{R}} A$ on $A$ factors as

$$(\int_{S^{n-1} \times \mathbb{R}} A) \wedge A \to A \wedge A \to A$$

\noindent via the map $\int_{S^{n-1} \times \mathbb{R}} A \to \int_{\mathbb{R}^n} A$ induced by the usual embedding $S^{n-1} \times \mathbb{R} \hookrightarrow \mathbb{R}^n$. For let $e: (S^{n-1} \times \mathbb{R}) \coprod \mathbb{R}^n \hookrightarrow \mathbb{R}^n$ be the embedding defining the module structure $(\int_{S^{n-1} \times \mathbb{R}} A) \wedge A \to A$. Then $e \times \mathbb{R}: (S^{n-1} \times \mathbb{R}^2) \coprod \mathbb{R}^{n+1} \hookrightarrow \mathbb{R}^{n+1}$ is homotopic to

$$(S^{n-1} \times \mathbb{R}^2) \coprod \mathbb{R}^{n+1} \hookrightarrow\mathbb{R}^{n+1} \coprod \mathbb{R}^{n+1} \hookrightarrow \mathbb{R}^{n+1}$$

Thus, the action $\Sigma^\infty _+ \Omega X \wedge A \to A$ factors through the usual left action $A \wedge A \to A$. As $\Sigma^\infty_+ \Omega X \to A$ is nullhomotopic though $E_1$-maps, the action is homotopically trivial.
\end{proof}

\begin{cor}\label{cor-triv-action}
If $f: X \to B^{n+1} G$ is an $E_1$-map and $X$ is $(n-1)$-connected, then
$$THH^{E_n}(A,A) \simeq A \wedge X_+$$
$$THC_{E_n}(A,A) \simeq \Map(X_+,A)$$
\end{cor}

\begin{rem}\label{rem-PD-triv}
If $X$ is additionally a closed manifold, the Poincar\'e duality of Corollary \ref{cor-atiyah-duality} between higher Hochschild homology and higher Hochschild cohomology becomes especially clear. The statement becomes usual Poincar\'e duality of homology with coefficients in $A$ and cohomology with coefficients in $A$.
\end{rem}

\begin{rem}\label{rem-higher-deligne}
If $A$ is an $E_{n+1}$-ring spectrum, $\Map(X_+,A)$ is an $E_{n+1}$-ring spectrum in a natural way. It would be interesting to know whether this agrees with the $E_{n+1}$ structure on higher Hochschild cohomology given by the higher Deligne conjecture.
\end{rem}

This recovers computations of topological Hochschild cohomology of $\mathbb{Z}/p$ (\cite{FLS}, 7.3) and $\mathbb{Z}_{(p)}$ (\cite{FP}):
$$THC(H\mathbb{Z}/p) \simeq \Map(\Omega S^3 _+, H\mathbb{Z}/p)$$
$$THC(H\mathbb{Z}_{(p)}) \simeq \Map(\Omega (S^3 \langle 3 \rangle)_+, H\mathbb{Z}_{(p)})$$

We can also calculate $THC_{E_2}$ of $H\mathbb{Z}/p$ and $H\mathbb{Z}_{(p)}$:

\begin{prop}\label{prop-E_2-EM}
$$THC_{E_2}(H\mathbb{Z}/p) \simeq \Map(S^3 _+, H\mathbb{Z}/p)$$
$$THC_{E_2}(H\mathbb{Z}_{(p)}) \simeq \Map(S^3 \langle 3 \rangle _+, H\mathbb{Z}_{(p)})$$
\end{prop}

\begin{proof}
For ease of notation, we will focus on $H\mathbb{Z}/p$; the proof for $H\mathbb{Z}_{(p)}$ is identical. 

As in the proof of Proposition \ref{prop-triv-action}, by the commutativity of $H\mathbb{Z}/p$, the action of $\int_{S^1 \times \mathbb{R}} H\mathbb{Z}/p$ on $H\mathbb{Z}/p$ factors through the left action of $H\mathbb{Z}/p$ on itself:

$$(\int_{S^1 \times \mathbb{R}} H\mathbb{Z}/p) \wedge H\mathbb{Z}/p \to H\mathbb{Z}/p \wedge H\mathbb{Z}/p \to H\mathbb{Z}/p$$

Again the commutativity of $H\mathbb{Z}/p$ implies that $\int_{S^1 \times \mathbb{R}} H\mathbb{Z}/p \to H\mathbb{Z}/p$ is a ring map. Composing this with the ring map $\Sigma^\infty_+ \Omega S^3 \to \int_{S^1 \times \mathbb{R}} H\mathbb{Z}/p$, we have a ring map $\Sigma^\infty_+ \Omega S^3 \to H\mathbb{Z}/p$ and the higher conjugation action factors as

$$\Sigma^\infty_+ \Omega S^3 \wedge H\mathbb{Z}/p \to H\mathbb{Z}/p \wedge H\mathbb{Z}/p \to H\mathbb{Z}/p$$

The action of $\Sigma^\infty_+ \Omega S^3$ determines a map $\Omega S^3 \to hAut(H\mathbb{Z}/p)$. As the action factors though the left action of $H\mathbb{Z}/p$ on itself, this map factors through $\Omega S^3 \to GL_1(H\mathbb{Z}/p)$. But $GL_1(H\mathbb{Z}/p)$ is homotopy equivalent to $(\mathbb{Z}/p)^{\times}$, therefore discrete, and $\Omega S^3$ is connected, so this map is nullhomotopic, and the action of $\Sigma^\infty _+ \Omega S^3$ on $H\mathbb{Z}/p$ is homotopically trivial.
\end{proof}

Proposition \ref{prop-triv-action} also allows us to calculate the $E_n$ topological Hochschild cohomology of cobordism spectra:

\begin{cor}\label{cor-triv-cob}
Let $G$ be one of the stabilized Lie groups ($O$,$U$,$SO$,...) or one of the stabilized discrete groups $\Sigma$ or $GL(\mathbb{Z})$. Then
$$THC_{E_n}(MG) \simeq \Map(B^{n+1} G _+, MG)$$
\end{cor}

Thus (higher) topological Hochschild cohomology of $MG$ is given by $G$-cobordism cohomology of spaces in the spectrum associated to $BG$. For example,
$$THC^*(MU) \cong MU^*(SU)$$ 
\noindent In general, higher topological Hochschild cohomology of $MU$ will alternate between $MU$-cohomology of connected covers of $U$ and $BU$.

\medskip

We can also calculate the topological Hochschild cohomology of the spectra $X(n)$. These spectra, introduced by Ravenel (see e.g. Section 6.5 of \cite{Rav}), provide a filtration of $MU$ and played an important role in the proof of the nilpotence theorem \cite{DHS}. The spectrum $X(n)$ is defined as the Thom spectrum of the map
$$\Omega SU(n) \to \Omega SU \simeq BU$$
\noindent Note that $X(n)$ is an $E_2$-ring spectrum, as it is the Thom spectrum of an $E_2$-map.

\begin{cor}\label{cor-triv-X(n)}
$$THC(X(n)) \simeq \Map(SU(n)_+, X(n))$$
\end{cor}

The Lie group $SU(n)$ has trivial tangent bundle, so this is equivalent to $\Sigma^{-d} SU(n)_+ \wedge X(n)$, where $d=\dim SU(n)$. Notice that this is indeed a shift, or twist by a trivial tangent bundle, of $THH(X(n))$; as in \cite{Bea}, $THH(X(n)) \simeq X(n) \wedge SU(n)_+$.

\bibliographystyle{plain}
\bibliography{Factorization_homology_of_Thom_spectra}

\end{document}